\documentclass[fleqn, pdftex]{article} 

\usepackage{amsmath,amssymb,amsthm,pifont,graphics,tikz}

\title{Arithmetical completeness theorems for monotonic modal logics}
\author{Haruka Kogure\footnote{Email: kogure1987@stu.kanazawa-u.ac.jp}
\footnote{College of Science and Engineering, School of Mathematics and Physics, Kanazawa University, Kakuma, Kanazawa 920-1192, Japan}
and Taishi Kurahashi\footnote{Email: kurahashi@people.kobe-u.ac.jp}
\footnote{Graduate School of System Informatics, Kobe University, 1-1 Rokkodai, Nada, Kobe 657-8501, Japan.}}
\date{}

\theoremstyle{plain}
\newtheorem{thm}{Theorem}[section]

\newtheorem{prop}[thm]{Proposition}
\newtheorem{cor}[thm]{Corollary}

\newtheorem{prob}[thm]{Problem}
\newtheorem{cl}[thm]{Claim}

\theoremstyle{definition}
\newtheorem{defn}[thm]{Definition}

\newtheorem{rem}[thm]{Remark}

\newcommand{\PA}{\mathsf{PA}}
\newcommand{\PR}{\mathrm{Pr}}
\newcommand{\Prf}{\mathrm{Prf}}
\newcommand{\Prov}{\mathrm{Prov}}
\newcommand{\Proof}{\mathrm{Proof}}

\newcommand{\gn}[1]{\ulcorner#1\urcorner}
\newcommand{\D}[1]{\mathbf{D#1}}
\newcommand{\B}[1]{\mathbf{B}_{#1}}
\newcommand{\M}{\mathbf{M}}
\newcommand{\SC}{\mathbf{\Sigma_1 C}}
\newcommand{\Fml}{\mathrm{Fml}_{\mathcal{L}_A}}
\newcommand{\num}{\overline}
\newcommand{\tto}{\twoheadrightarrow}
\newcommand{\N}{\mathbb{N}}

\newcommand{\LA}{\mathcal{L}_A}
\newcommand{\LB}{\mathcal{L}(\Box)}

\newcommand{\GL}{\mathsf{GL}}
\newcommand{\Sub}{\mathsf{Sub}}

\newcommand{\MN}{\mathsf{MN}}
\newcommand{\MNF}{\mathsf{MN4}}
\newcommand{\MNP}{\mathsf{MNP}}
\newcommand{\MND}{\mathsf{MND}}
\newcommand{\MNPF}{\mathsf{MNP4}}
\newcommand{\MNDF}{\mathsf{MND4}}

\newcommand{\PL}{\mathsf{PL}}

\begin{document}

\maketitle

\begin{abstract}
We investigate modal logical aspects of provability predicates $\PR_T(x)$ satisfying the following condition: \\
$\M$: If $T \vdash \varphi \to \psi$, then $T \vdash \PR_T(\gn{\varphi}) \to \PR_T(\gn{\psi})$. 

We prove the arithmetical completeness theorems for monotonic modal logics $\MN$, $\MNF$, $\MNP$, $\MNPF$, and $\MND$ with respect to provability predicates satisfying the condition $\M$. 
That is, we prove that for each logic $L$ of them, there exists a $\Sigma_1$ provability predicate $\PR_T(x)$ satisfying $\M$ such that the provability logic of $\PR_T(x)$ is exactly $L$. 
In particular, the modal formulas $\mathrm{P}$: $\neg \Box \bot$ and $\mathrm{D}$: $\neg (\Box A \land \Box \neg A)$ are not equivalent over non-normal modal logic and correspond to two different formalizations $\neg \PR_T(\gn{0=1})$ and $\neg \big(\PR_T(\gn{\varphi}) \land \PR_T(\gn{\neg \varphi}) \bigr)$ of consistency statements, respectively. 
Our results separate these formalizations in terms of modal logic. 
\end{abstract}

\section{Introduction}

In the usual proof of G\"odel's incompleteness theorems, provability predicates of a suitable theory $T$, that is, formulas weakly representing the set of all theorems of $T$, play important roles. 
In particular, a significant step in the proof of the second incompleteness theorem is to prove that a $\Sigma_1$ canonical provability predicate $\Prov_T(x)$ of $T$ satisfies the following Hilbert--Bernays--L\"ob's derivability conditions: 
\begin{description}
	\item [D1:] If $T \vdash \varphi$, then $T \vdash \Prov_T(\gn{\varphi})$. 
	\item [D2:] $T \vdash \Prov_T(\gn{\varphi \to \psi}) \to \bigl(\Prov_T(\gn{\varphi}) \to \Prov_T(\gn{\psi}) \bigr)$. 
	\item [D3:] $T \vdash \Prov_T(\gn{\varphi}) \to \Prov_T(\gn{\Prov_T(\gn{\varphi})})$. 
\end{description}
Then, it is shown that $T \nvdash \neg \Prov_T(\gn{0=1})$ if $T$ is consistent. 
Interestingly, the above conditions match modal logic by interpreting the modal operator $\Box$ as $\Prov_T(x)$. 
Then, many modal logical investigations of provability predicates have been made. 
Among other things, one of the most important progress in this study is Solovay's arithmetical completeness theorem \cite{Sol}. 
Solovay's theorem states that for any $\Sigma_1$-sound recursively enumerable extension $T$ of Peano Arithmetic $\PA$, the set of all $T$-verifiable modal principles is characterized by the normal modal logic $\GL$. 
Solovay's proof is carried out by embedding finite Kripke models appropriate to $\GL$ into arithmetic. 

On the other hand, there are non-canonical provability predicates that only partially satisfy the derivability conditions. 
Kurahashi \cite{Kur20_1} systematically studies the dependencies between various derivability conditions and several versions of the second incompleteness theorem.  
Among various non-canonical provability predicates, Rosser provability predicate $\PR_T^{\mathrm{R}}(x)$ of $T$ is particularly important in studying the incompleteness theorems, which was essentially introduced by Rosser \cite{Ros} to improve G\"odel's first incompleteness theorem (cf.~\cite{Kre}). 
Rosser provability predicate is useful to investigate the limitation of the second incompleteness theorem because it is known that the second incompleteness theorem does not hold for $\PR_T^{\mathrm{R}}(x)$, that is, $\PA \vdash \neg \PR_T^{\mathrm{R}}(\gn{0=1})$ holds. 
Hence, by the proof of the second incompleteness theorem, $\PR_T^{\mathrm{R}}(x)$ does not satisfy at least one of the conditions \textbf{D2} and \textbf{D3}. 

The derivability conditions for Rosser provability predicates have been studied by many authors (see \cite{Ara,BM,GS,KT,Kur20_1,Kur21,Sha}). 
In particular, these studies have shown that whether a Rosser provability predicate satisfies $\D{2}$ or $\D{3}$ depends on the choice of a predicate.
The existence of a Rosser provability predicate satisfying $\D{2}$ was proved by Bernardi and Montagna \cite{BM} and Arai \cite{Ara}, and the existence of a Rosser provability predicate satisfying $\D{3}$ was proved by Arai. 
By these observations, it is obtained that the second incompleteness theorem cannot be proved only by $\D{2}$ or $\D{3}$. 

If a provability predicate $\PR_T(x)$ satisfies $\D{2}$, then it is easily shown that for any formula $\varphi$ of arithmetic, $\neg \PR_T(\gn{0=1})$ and $\neg \bigl(\PR_T(\gn{\varphi}) \land \PR_T(\gn{\neg \varphi}) \bigr)$ are $T$-provably equivalent. 
Hence, if a Rosser provability predicate $\PR_T^{\mathrm{R}}(x)$ satisfies $\D{2}$, then $T$ also proves $\neg \bigl(\PR_T^{\mathrm{R}}(\gn{\varphi}) \land \PR_T^{\mathrm{R}}(\gn{\neg \varphi}) \bigr)$ for any $\varphi$. 
Here, we focus on the following condition $\M$. 
\begin{description}
	\item [$\M$:] If $T \vdash \varphi \to \psi$, then $T \vdash \PR_T(\gn{\varphi}) \to \PR_T(\gn{\psi})$. 
\end{description}
The condition $\M$ originates from Hilbert and Bernays \cite{HB} and it is adopted as the first condition of their derivability conditions. 
This condition is called $\B{2}$ in \cite{Kur20_2,Kur21} and is also considered in \cite{Mon, Vis}. 
It is proved in \cite{Kur21} that if a provability predicate $\PR_T(x)$ satisfies $\M$ and $\D{3}$, then $T \nvdash \neg \bigl(\PR_T(\gn{\varphi}) \land \PR_T(\gn{\neg \varphi}) \bigr)$ for some $\varphi$ whenever $T$ is consistent. 
This is a version of the second incompleteness theorem. 
Furthermore, it is proved in \cite{Kur21} that there exists a Rosser provability predicate satisfying both $\M$ and $\D{3}$, and hence there is a difference between the unprovability of the two consistency statements $\neg \PR_T(\gn{0=1})$ and $\neg \bigl(\PR_T(\gn{\varphi}) \land \PR_T(\gn{\neg \varphi}) \bigr)$. 

A modal logical study of Rosser provability predicates was initiated by Guaspari and Solovay \cite{GS}. 
They developed a modal logic $\mathsf{R}$ dealing with the more general notion of witness comparison. 
Shavrukov \cite{Sha} introduced a bimodal logic $\mathsf{GR}$, which explicitly deals with both usual and Rosser's predicates, and proved its arithmetical completeness theorem.
Kurahashi \cite{Kur20_1} investigated Rosser provability predicates satisfying $\D{2}$ in terms of modal logic.
The modal logic corresponding to such a Rosser provability predicate is a normal modal logic containing $\mathsf{KD}$, and by applying Solovay's proof method, the existence of a Rosser provability predicate exactly corresponding to $\mathsf{KD}$ was proved.

However, as mentioned above, provability predicates satisfying $\D{2}$ do not distinguish between $\neg \PR_T(\gn{0=1})$ and $\neg \bigl(\PR_T(\gn{\varphi}) \land \PR_T(\gn{\neg \varphi}) \bigr)$. 
Then, can we study provability predicates that do not satisfy $\D{2}$ in terms of modal logic?
The modal logic corresponding to such a provability predicate does not contain the weakest normal modal logic $\mathsf{K}$, and thus Kripke semantics would not work well for it. 
Thus, Solovay's proof technique of embedding Kripke models into arithmetic would not be directly applicable to such a logic. 
Recently, the second author of the present paper has attempted to extend the proof method of Solovay's theorem to non-normal modal logics. 
In \cite{Kur}, focusing on the fact that the pure logic of necessitation $\mathsf{N}$ introduced by Fitting, Marek, and Truszczy\'nski \cite{FMT} has a relational semantics similar to Kripke semantics, the arithmetical completeness theorems of $\mathsf{N}$ and several extensions of $\mathsf{N}$ are proved by embedding models corresponding to these logics into arithmetic.

We now turn our attention to logics satisfying the rule \textsc{RM} $\dfrac{A \to B}{\Box A \to \Box B}$ corresponding to the condition $\M$. 
The purpose of the present paper is to extend Solovay's proof method to extensions of the monotonic modal logic $\MN$ having the inference rules Necessitation and \textsc{RM}. 
Such logics are called \textit{monotonic modal logics}, and in particular it is known that monotonic neighborhood semantics works well for these logics (cf.~Chellas \cite{Che}). 
We prove the arithmetical completeness theorems for $\MN$ with respect to provability predicates satisfying $\M$. 
Furthermore, we prove the arithmetical completeness of the logic $\MNF = \MN + (\Box A \to \Box \Box A)$ corresponding to provability predicates satisfying $\M$ and $\D{3}$. 
Also, in monotonic modal logics, the two different consistency statements as above correspond to the two different axiom schemata $\mathrm{P}$: $\neg \Box \bot$ and $\mathrm{D}$: $\neg (\Box A \land \Box \neg A)$, respectively. 
We prove the arithmetical completeness theorems with respect to Rosser provability predicates for the logics $\MNP$ and $\MND$ obtained by adding $\mathrm{P}$ and $\mathrm{D}$ to $\MN$, respectively. 
These results show that the above two different consistency statements can be separated in terms of modal logic.
We also prove the arithmetical completeness theorem for the logic $\MNPF$ with respect to Rosser provability predicates satisfying $\M$ and $\D{3}$ whose existence is proved in \cite{Kur21}.

This paper is organized as follows. 
In Section \ref{Sec:PP}, we introduce basic notions on provability predicates and modal logics corresponding to provability predicates, and survey previous research. 
In Section \ref{Sec:MML}, we introduce the monotonic modal logics $\MN$, $\MNF$, $\MNP$, $\MND$, $\MNPF$, and $\MNDF$ and the notions of $\MN$-frames and $\MN$-models. 
Then, we prove that these logics have the finite frame property with respect to $\MN$-frames. 
Sections from \ref{Sec:MN} to \ref{Sec:MND} are devoted to proving the arithmetical completeness theorems of the logics $\MN$, $\MNF$, $\MNP$, $\MNPF$, and $\MND$. 
Finally, in Section \ref{Sec:FW}, we discuss future work.

\section{Provability predicates}\label{Sec:PP}

Throughout the present paper, $T$ always denotes a primitive recursively axiomatized consistent extension of Peano Arithmetic $\PA$ in the language $\LA$ of first-order arithmetic. 
Let $\omega$ denote the set of all natural numbers. 
For each $n \in \omega$, the numeral for $n$ is denoted by $\num{n}$. 
We fix some standard G\"odel numbering, and for each $\LA$-formula $\varphi$, let $\gn{\varphi}$ be the numeral for the G\"odel number of $\varphi$. 
We may assume that our G\"odel numbering is monotone. 
Namely, if $\alpha$ is a proper sub-expression of a finite sequence $\beta$ of $\LA$-symbols, then the G\"odel number of $\alpha$ is smaller than that of $\beta$. 
Let $\langle \xi_t \rangle_{t \in \omega}$ be the repetition-free primitive recursive enumeration of all $\LA$-formulas in ascending order of G\"odel numbers. 

We say that an $\LA$-formula $\PR_T(x)$ is a \textit{provability predicate} of $T$ if it weakly represents the set of all theorems of $T$ in $\PA$, namely, for any $n \in \omega$, $\PA \vdash \PR_T(\num{n})$ if and only if $n$ is the G\"odel number of a theorem of $T$. 
Let $\Proof_T(x, y)$ be a primitive recursive $\LA$-formula naturally expressing that $y$ is the G\"odel number of a $T$-proof of a formula whose G\"odel number is $x$. 
Let $\Prov_T(x)$ be the $\Sigma_1$ formula $\exists y \Proof_T(x, y)$, then it is shown that $\Prov_T(x)$ is a provability predicate of $T$. 

The notion of Rosser provability predicates was essentially introduced by Rosser \cite{Ros} to improve G\"odel's first incompleteness theorem. 
Also, Rosser provability predicates are useful to investigate G\"odel's second incompleteness theorem because the theorem does not hold for Rosser provability predicates. 
We say that a $\Sigma_1$ formula $\PR_T^{\mathrm{R}}(x)$ is a \textit{Rosser provability predicate} of $T$ if it is of the form
\[
	\exists y \bigl(\Fml(x) \land \Prf_T(x, y) \land \forall z < y\, \neg \Prf_T(\dot{\neg}(x), z) \bigr)
\]
for some primitive recursive $\LA$-formula $\Prf_T(x, y)$ satisfying the following conditions: 
\begin{enumerate}
	\item For any $\varphi$ and $n \in \omega$, $\PA \vdash \Prf_T(\gn{\varphi}, \num{n}) \leftrightarrow \Proof_T(\gn{\varphi}, \num{n})$, 
	\item $\PA \vdash \forall x \Bigl(\Fml(x) \to \bigl(\exists y \Prf_T(x, y) \leftrightarrow \Prov_T(x) \bigr) \Bigr)$. 
\end{enumerate}
Here, $\dot{\neg}(x)$ is a primitive recursive term corresponding to a primitive recursive function calculating the G\"odel number of $\neg \varphi$ from that of $\varphi$, and $\Fml(x)$ is a primitive recursive formula naturally expressing that $x$ is the G\"odel number of an $\LA$-formula. 
It is shown that Rosser provability predicates are $\Sigma_1$ provability predicates of $T$.

\subsection{Derivability conditions}

In connection with the second incompleteness theorem, the various conditions that provability predicates are expected to satisfy are known as derivability conditions.

\begin{defn}[Derivability conditions]
\leavevmode
\begin{description}
	\item [D1:] If $T \vdash \varphi$, then $T \vdash \PR_T(\gn{\varphi})$. 
	\item [D2:] $T \vdash \PR_T(\gn{\varphi \to \psi}) \to \bigl(\PR_T(\gn{\varphi}) \to \PR_T(\gn{\psi}) \bigr)$. 
	\item [D3:] $T \vdash \PR_T(\gn{\varphi}) \to \PR_T(\gn{\PR_T(\gn{\varphi})})$. 
	\item [$\Sigma_1$C:] If $\varphi$ is a $\Sigma_1$ sentence, then $T \vdash \varphi \to \PR_T(\gn{\varphi})$. 
	\item [M:] If $T \vdash \varphi \to \psi$, then $T \vdash \PR_T(\gn{\varphi}) \to \PR_T(\gn{\psi})$. 
\end{description}
\end{defn}
Every provability predicate satisfies $\D{1}$. 
If a $\Sigma_1$ provability predicate satisfies $\SC$, then it also satisfies $\D{3}$. 
Also, the condition $\M$ is weaker than $\D{2}$. 
It is shown that the provability predicate $\Prov_T(x)$ satisfies all of these conditions. 

There are various formulations of the statement that the theory $T$ is consistent, and we focus on two of them.
The first one is of the form $\neg \PR_T(\gn{0=1})$, which is widely used in the literature dealing with the second incompleteness theorem. 
The second one is the schematic consistency statement 
\[
	\{\neg \bigl(\PR_T(\gn{\varphi}) \land \PR_T(\gn{\neg \varphi}) \bigr) \mid \varphi\ \text{is an}\ \LA\text{-formula}\}, 
\]
that is introduced in \cite{Kur21}. 
Notice that for each $\LA$-formula $\varphi$, if $\PR_T(x)$ satisfies $\D{2}$, then $\neg \PR_T(\gn{0=1})$ and $\neg \bigl(\PR_T(\gn{\varphi}) \land \PR_T(\gn{\neg \varphi}) \bigr)$ are $T$-provably equivalent. 
Among other things, we introduce the following three versions of the second incompleteness theorem. 

\begin{thm}[The second incompleteness theorem]\label{G2}
Let $\PR_T(x)$ be a provability predicate of $T$. 
\begin{enumerate}
	\item \textup{(L\"ob \cite{Lob})} If $\PR_T(x)$ satisfies $\D{2}$ and $\D{3}$, then $T \nvdash \neg \PR_T(\gn{0=1})$. 
	Furthermore, $T \vdash \PR_T(\gn{\PR_T(\gn{\varphi}) \to \varphi}) \to \PR_T(\gn{\varphi})$. 
	\item \textup{(Jeroslow \cite{Jer}; Kreisel and Takeuti \cite{KT})} If $\PR_T(x)$ is $\Sigma_1$ and satisfies $\SC$, then for some $\LA$-sentence $\varphi$, $T \nvdash \neg \bigl(\PR_T(\gn{\varphi}) \land \PR_T(\gn{\neg \varphi}) \bigr)$. 
	\item \textup{(Kurahashi \cite{Kur20_2,Kur21})} If $\PR_T(x)$ satisfies $\M$ and $\D{3}$, then for some $\LA$-sentence $\varphi$, $T \nvdash \neg \bigl(\PR_T(\gn{\varphi}) \land \PR_T(\gn{\neg \varphi}) \bigr)$. 
\end{enumerate}
\end{thm}

Not all provability predicates satisfy the second incompleteness theorem. 
For example, Mostowski \cite{Mos} showed that there exists a $\Sigma_1$ provability predicate $\PR_T(x)$ of $T$ satisfying $\SC$ such that $\PA \vdash \neg \PR_T(\gn{0=1})$. 
This fact together with clause 2 of Theorem \ref{G2} implies that the two consistency statements $\neg \PR_T(\gn{0=1})$ and $\neg \bigl(\PR_T(\gn{\varphi}) \land \PR_T(\gn{\neg \varphi}) \bigr)$ are different in general, and the conclusion of clause 2 of Theorem \ref{G2} cannot be strengthened to $T \nvdash \neg \PR_T(\gn{0=1})$. 

A more prominent example of provability predicates for which the second incompleteness theorem does not hold are Rosser provability predicates, that is, it is shown that $\PA$ proves $\neg \PR_T^{\mathrm{R}}(\gn{0=1})$ for any Rosser provability predicate $\PR_T^{\mathrm{R}}(x)$ of $T$ (cf.~\cite{Kre}). 
Hence, at least one of $\D{2}$ and $\D{3}$ does not hold for $\PR_T^{\mathrm{R}}(x)$. 
Bernardi and Montagna \cite{BM} and Arai \cite{Ara} proved that there exists a Rosser provability predicate $\PR_T^{\mathrm{R}}(x)$ of $T$ satisfying $\D{2}$. 
For such $\PR_T^{\mathrm{R}}(x)$ and any $\LA$-formula $\varphi$, $T \vdash \neg \bigl(\PR_T^{\mathrm{R}}(\gn{\varphi}) \land \PR_T^{\mathrm{R}}(\gn{\neg \varphi}) \bigr)$ holds.  
On the other hand, Arai \cite{Ara} proved that there exists a Rosser provability predicate $\PR_T^{\mathrm{R}}(x)$ of $T$ satisfying $\D{3}$. 
These observations indicate that one of $\D{2}$ and $\D{3}$ cannot be dropped from the assumption of clause 1 of Theorem \ref{G2}. 
Moreover, the second author \cite{Kur21} proved that there exists a Rosser provability predicate of $T$ satisfying $\M$ and $\D{3}$. 
Thus, the conclusion of clause 3 of Theorem \ref{G2} cannot be strengthened to $T \nvdash \neg \PR_T(\gn{0=1})$.

\subsection{Modal logics of provability predicates}

The language $\LB$ of modal propositional logic consists of countably many propositional variables $p_0, p_1, \ldots$, the logical constant $\bot$, the logical connective $\to$, and the modal operator $\Box$. 
Other symbols such as $\top, \neg, \land, \lor$, and $\Diamond$ are introduced as abbreviations in the usual way. 

We say that a modal logic $L$ is \textit{normal} if it contains all tautologies in the language $\LB$ and the distribution axiom scheme $\Box (A \to B) \to (\Box A \to \Box B)$ and is closed under Modus Ponens (\textsc{MP}) $\dfrac{A\to B \quad A}{B}$, Necessitation (\textsc{Nec}) $\dfrac{A}{\Box A}$, and uniform substitution. 
The weakest normal modal logic is called $\mathsf{K}$. 
The normal modal logics $\GL$ and $\mathsf{KD}$ are obtained from $\mathsf{K}$ by adding the axiom schemata $\Box(\Box A \to A) \to \Box A$ and $\neg \Box \bot$, respectively. 

For each provability predicate $\PR_T(x)$ of $T$, we say that a mapping $f$ from $\LB$-formulas to $\LA$-sentences is an \textit{arithmetical interpretation} based on $\PR_T(x)$ if it satisfies the following clauses: 
\begin{itemize}
	\item $f(\bot)$ is $0=1$, 
	\item $f(A \to B)$ is $f(A) \to f(B)$, 
	\item $f(\Box A)$ is $\PR_T(\gn{f(A)})$. 
\end{itemize}
The \textit{provability logic} $\PL(\PR_T)$ of $\PR_T(x)$ is the set of all $\LB$-formulas $A$ satisfying that $T$ proves $f(A)$ for any arithmetical interpretation $f$ based on $\PR_T(x)$. 
It is shown that if $\PR_T(x)$ satisfies $\D{2}$, then $\PL(\PR_T)$ is a normal modal logic. 

A pioneering result in the research of provability logics is Solovay's arithmetical completeness theorem \cite{Sol} (see also \cite{AB,Boo,JD}). 
Solovay proved that if $T$ is $\Sigma_1$-sound, then $\PL(\Prov_T)$ is exactly the logic $\GL$. 
Solovay proved his theorem by defining a computable function for a given finite Kripke model of $\GL$, the so-called Solovay function by referring to its accessibility relation, and embedding the model into arithmetic. 
Solovay's proof method can be used to prove the arithmetical completeness theorems for other normal modal logics.
It is proved in \cite{Kur} that there exists a $\Sigma_1$ provability predicate $\PR_T(x)$ of $T$ such that $\PL(\PR_T) = \mathsf{K}$ (see also \cite{Kur18_1,Kur18_2}). 
Also, it is proved in \cite{Kur20_1} that there exists a Rosser provability predicate $\PR_T^{\mathrm{R}}(x)$ of $T$ such that $\PL(\PR_T^{\mathrm{R}}) = \mathsf{KD}$. 

As mentioned above, the two consistency statements $\neg \PR_T(\gn{0=1})$ and $\{\neg \bigl(\PR_T(\gn{\varphi}) \land \PR_T(\gn{\neg \varphi}) \bigr) \mid \varphi$ is an $\LA$-formula$\}$ are $T$-provably equivalent if $\PR_T(x)$ satisfies $\D{2}$, but not in general. 
Thus, the difference between these two consistency statements cannot be captured in the framework of normal modal logics. 


\section{Monotonic modal logics}\label{Sec:MML}

It is easily shown that every normal modal logic is closed under the rule \textsc{RM} $\dfrac{A \to B}{\Box A \to \Box B}$ of monotonicity. 
Notice that \textsc{Nec}, the distribution axiom scheme, the axiom scheme $\Box A \to \Box \Box A$, and \textsc{RM} are modal counterparts of the derivability conditions $\D{1}$, $\D{2}$, $\D{3}$, and $\M$, respectively. 
Our purpose of the present paper is to investigate provability predicates satisfying $\M$, and so we shall deal with logics that are closed under the rule \textsc{RM}. 
Such logics are called \textit{monotonic}. 
Thus, all normal modal logics are monotonic. 

We do not necessarily require the distribution axiom scheme because provability predicate satisfying $\M$ do not necessarily satisfy $\D{2}$.
On the other hand, since every provability predicate satisfies $\D{1}$, we require \textsc{Nec}. 
The weakest modal logic satisfying our requirements is called $\MN$ and is axiomatized as follows: 
The axioms of $\MN$ are only tautologies in the language $\LB$. 
The inference rules of $\MN$ are \textsc{MP}, \textsc{Nec}, and the rule \textsc{RM}.

The logic $\MN$ is strictly weaker than the weakest normal modal logic $\mathsf{K}$, and so Kripke semantics does not work well for $\MN$. 
On the other hand, extensions of $\MN$ are closed under the rule \textsc{RE} $\dfrac{A \leftrightarrow B}{\Box A \leftrightarrow \Box B}$, and it is known that neighborhood semantics provides a semantics alternative to Kripke semantics for such logics (cf.~Chellas \cite{Che}). 
We introduce the following relational semantics for extensions of $\MN$, which is in fact equivalent to the usual neighborhood semantics validating $\MN$ (see Remark \ref{Neighborhood} below). 
By adopting a relational semantics similar to Kripke semantics, we expect to understand how to extend Solovay's proof method to monotonic modal logics.
Our relational semantics is very similar to what is well known in the field of interpretability logic as generalized Veltman semantics, that was introduced by Verbrugge \cite{Ver} (see \cite{JRMV} for more detail).

\begin{defn}[$\MN$-frames and models]
We say that a tuple $(W, \prec)$ is an \textit{$\MN$-frame} if $W$ is a non-empty set and $\prec$ is a binary relation between $W$ and $(\mathcal{P}(W) \setminus \{\emptyset\})$ satisfying the following condition: 
\begin{description}
	\item [Monotonicity:] $(\forall x \in W) (\forall U, V \in \mathcal{P}(W)) (x \prec V\ \&\ V \subseteq U \Rightarrow x \prec U)$. 
\end{description}
We say that a triple $(W, \prec, \Vdash)$ is an \textit{$\MN$-model} if $(W, \prec)$ is an $\MN$-frame and $\Vdash$ is a binary relation between $W$ and the set of all formulas satisfying the usual conditions for satisfaction and the following condition: 
\begin{itemize}
	\item $x \Vdash \Box A \iff (\forall V \in \mathcal{P}(W)) \bigl(x \prec V \Rightarrow (\exists y \in V) (y \Vdash A) \bigr)$. 
\end{itemize}
A formula $A$ is said to be \textit{valid} in an $\MN$-frame $(W, \prec)$ if $x \Vdash A$ for all satisfaction relations $\Vdash$ on the frame and all $x \in W$. 
\end{defn}

It is easily shown that for each $\MN$-frame $(W, \prec)$, the set of all formulas valid in $(W, \prec)$ is closed under \textsc{MP}, \textsc{Nec}, uniform substitution, and \textsc{RM}. 
Therefore, every theorem of $\MN$ is valid in all $\MN$-frames. 

\begin{rem}\label{Neighborhood}
A monotonic neighborhood frame is a tuple $(W, \delta)$, where $W$ is a non-empty set and $\delta$ is a mapping $\mathcal{P}(W) \to \mathcal{P}(W)$ satisfying the following two conditions:
\begin{enumerate}
	\item $\delta(\emptyset) = \emptyset$, 
	\item For any $U, V \subseteq W$, if $U \subseteq V$, then $\delta(U) \subseteq \delta(V)$. 
\end{enumerate}
A monotonic neighborhood model is a triple $(W, \delta, v)$, where $(W, \delta)$ is a monotonic neighborhood frame and $v$ is a mapping from $\LB$-formulas into subsets of $W$ satisfying the following conditions: 
\begin{enumerate}
	\item $v(\bot) = \emptyset$, 
	\item $v(A \to B) = (W \setminus v(A)) \cup v(B)$, 
	\item $v(\Diamond A) = \delta(v(A))$. 
\end{enumerate}
It is shown that $\MN$-frames and monotonic neighborhood frames are transformable into each other through the equivalence $x \prec V \iff x \in \delta(V)$. 
In fact, if $(W, \prec, \Vdash)$ and $(W, \delta, v)$ satisfy this equivalence and the equivalence $x \Vdash p \iff x \in v(p)$ for all $x \in W$ and propositional variables $p$, then $x \Vdash A \iff x \in v(A)$ holds for all $x \in W$ and $\LB$-formulas $A$. 
This is proved by induction on the construction of $A$, and we give a proof of the case that $A$ is of the form $\Diamond B$: 
\begin{align*}
	x \Vdash \Diamond B & \iff (\exists V \in \mathcal{P}(W)) \bigl(x \prec V\ \&\ (\forall y \in V) (y \Vdash B) \bigr), \\
	& \stackrel{\textrm{I.H.}}{\iff}  (\exists V \in \mathcal{P}(W)) \bigl(x \prec V\ \&\ V \subseteq v(B) \bigr), \\
	& \stackrel{(\ast)}{\iff} x \prec v(B), \\
	& \iff x \in \delta(v(B)), \\
	& \iff x \in v(\Diamond B). 
\end{align*}
Notice that in the equivalence $(\ast)$, the monotonicity of $(W, \prec)$ is used. 
In this sense, our relational semantics based on $\MN$-frames is equivalent to monotonic neighborhood semantics. 
\end{rem}

We also deal with extensions of $\MN$. 
The logics $\MNP$ and $\MND$ are obtained from $\MN$ by adding the axiom schemata $\mathrm{P}$: $\neg \Box \bot$ and $\mathrm{D}$: $\neg (\Box A \land \Box \neg A)$, respectively. 
Notice that the axiom schemata $\mathrm{P}$ and $\mathrm{D}$ are equivalent over $\mathsf{K}$, but this is not the case in general. 
This fact corresponds to the fact that there exists a Rosser provability predicate $\PR_T^{\mathrm{R}}(x)$ of $T$ such that $T \nvdash \neg \bigl(\PR_T^{\mathrm{R}}(\gn{\varphi}) \land \PR_T^{\mathrm{R}}(\gn{\neg \varphi}) \bigr)$ for some $\LA$-formula $\varphi$. 
Let $\MNF$, $\MNPF$, and $\MNDF$ be logics obtained from $\MN$, $\MNP$, and $\MND$ by adding the axiom scheme $\Box A \to \Box \Box A$, respectively. 
It is easily shown that $\MNP$ is deductively equivalent to the logic obtained by adding the rule $\dfrac{\neg A}{\neg \Box A}$ into $\MN$.

As in Kripke semantics, the validity of each formula may be characterized as a property of binary relations of $\MN$-frames. 

\begin{defn}[$\MNP$-frames and $\MND$-frames]
Let $(W, \prec)$ be an $\MN$-frame. 
\begin{itemize}
	\item We say $(W, \prec)$ is \textit{transitive} if for any $x \in W$ and $V \in \mathcal{P}(W)$, if $x \prec V$ and for each $y \in V$, there is $U_y \in \mathcal{P}(W)$ such that $y \prec U_y$, then $x \prec \bigcup_{y \in V} U_y$. 
	\item We say $(W, \prec)$ is an \textit{$\MNP$-frame} if for any $x \in W$, there exists a $V \in \mathcal{P}(W)$ such that $x \prec V$. 
	\item We say $(W, \prec)$ is an \textit{$\MND$-frame} if for any $x \in W$ and $V \in \mathcal{P}(W)$, $x \prec V$ or $x \prec (W \setminus V)$. 
\end{itemize}
\end{defn}

\begin{prop}\label{FC}
Let $(W, \prec)$ be any $\MN$-frame.  
\begin{enumerate}
	\item $\Box p \to \Box \Box p$ is valid in $(W, \prec)$ if and only if $(W, \prec)$ is transitive. 
	\item $\neg \Box \bot$ is valid in $(W, \prec)$ if and only if $(W, \prec)$ is an $\MNP$-frame. 
	\item $\neg (\Box p \land \Box \neg p)$ is valid in $(W, \prec)$ if and only if $(W, \prec)$ is an $\MND$-frame. 
\end{enumerate}
\end{prop}
\begin{proof}
1. $(\Rightarrow)$: We prove the contrapositive. 
Suppose that $(W, \prec)$ is not transitive, that is, there exist $x$, $V$, and $\{U_y\}_{y \in V}$ such that $x \prec V$, $y \prec U_y$ for all $y \in V$, and $\neg \bigl(x \prec \bigcup_{y \in V} U_y \bigr)$. 
Let $\Vdash$ be a satisfaction relation on $(W, \prec)$ satisfying that for any $u \in W$, $u \Vdash p \iff u \notin \bigcup_{y \in V} U_y$. 
Let $V'$ be any subset satisfying $x \prec V'$. 
Then, by the monotonicity, $V' \nsubseteq \bigcup_{y \in V} U_y$. 
That is, $z \notin \bigcup_{y \in V} U_y$ for some $z \in V'$. 
Hence, $z \Vdash p$ for some $z \in V'$. 
We obtain that $x \Vdash \Box p$. 

For each $y_0 \in V$, it follows from $U_{y_0} \subseteq \bigcup_{y \in V} U_y$ that for any $z \in U_{y_0}$, $z \nVdash p$. 
Since $y_0 \prec U_{y_0}$, we have $y_0 \nVdash \Box p$. 
Therefore, $x \nVdash \Box \Box p$. 
We conclude $x \nVdash \Box p \to \Box \Box p$. 

$(\Leftarrow)$: Suppose that $(W, \prec)$ is transitive. 
Let $\Vdash$ be any satisfaction relation on $(W, \prec)$ and $x \in W$ be such that $x \nVdash \Box \Box p$. 
Then, there exists a $V$ such that $x \prec V$ and for all $y \in V$, $y \nVdash \Box p$. 
Also, for each $y \in V$, there exists $U_y$ such that $y \prec U_y$ and for all $z \in U_y$, $z \nVdash p$. 
Then, for any $z \in \bigcup_{y \in V} U_y$, $z \nVdash p$. 
Since $x \prec \bigcup_{y \in V} U_y$, we obtain $x \nVdash \Box p$. 
We conclude that $\Box p \to \Box \Box p$ is valid in $(W, \prec)$. 

2. This is verified by the following equivalence: 
\begin{align*}
	x \Vdash \neg \Box \bot & \iff x \nVdash \Box \bot, \\
	& \iff (\exists V \in \mathcal{P}(W)) \bigl(x \prec V\ \&\ (\forall y \in V)(y \nVdash \bot) \bigr), \\
	& \iff (\exists V \in \mathcal{P}(W)) (x \prec V). 
\end{align*}

3. $(\Rightarrow)$: Suppose that $\neg (\Box p \land \Box \neg p)$ is valid in $(W, \prec)$. 
Let $x \in W$ and $V \in \mathcal{P}(W)$. 
Let $\Vdash$ be a satisfaction relation on $(W, \prec)$ satisfying that for each $u \in W$, $u \Vdash p \iff u \in V$.
Since $x \Vdash \neg (\Box p \land \Box \neg p)$, we have $x \nVdash \Box p$ or $x \nVdash \Box \neg p$. 
\begin{itemize}
	\item If $x \nVdash \Box p$, then there exists a $U \in \mathcal{P}(W)$ such that $x \prec U$ and $(\forall y \in U)(y \nVdash p)$. 
	This means $U \subseteq W \setminus V$. 
	By the monotonicity, we obtain $x \prec (W \setminus V)$. 
	\item If $x \nVdash \Box \neg p$, then there exists a $U \in \mathcal{P}(W)$ such that $x \prec U$ and $(\forall y \in U)(y \Vdash p)$. 
	Then, $U \subseteq V$, and hence $x \prec V$. 
\end{itemize}
We have shown that $(W, \prec)$ is an $\MND$-frame. 

$(\Leftarrow)$: 
Suppose that $(W, \prec)$ is an $\MND$-frame. 
Let $x \in W$ and $\Vdash$ be an arbitrary satisfaction relation on $(W, \prec)$. 
Let $V$ be the set $\{y \in W \mid y \Vdash p\}$, then we have $x \prec V$ or $x \prec (W \setminus V)$. 
If $x \prec V$, then $x \nVdash \Box \neg p$, and if $x \prec (W \setminus V)$, then $x \nVdash \Box p$. 
In either case, we obtain $x \Vdash \neg (\Box p \land \Box \neg p)$. 
\end{proof}

Let transitive $\MN$-frame, transitive $\MNP$-frame, and transitive $\MND$-frame be called $\MNF$-frame, $\MNPF$-frame, and $\MNDF$-frame, respectively.

We are ready to prove the finite frame property of the logics $\MN$, $\MNF$, $\MNP$, $\MND$, $\MNPF$, and $\MNDF$. 
Let $A$ be an arbitrary $\LB$-formula. 
If $A$ is of the form $\neg B$, then let ${\sim} A$ be $B$; otherwise, ${\sim} A$ denotes $\neg A$. 
Let $\Sub(A)$ be the set of all subformulas of $A$. 
We define $\num{\Sub(A)}$ to be the union of the sets $\Sub(A)$, $\{{\sim} B \mid B \in \Sub(A)\}$, and $\{\Box \bot, \neg \Box \bot, \Box \top, \neg \Box \top, \top, \bot\}$. 
Let $X$ be a finite set of $\LB$-formulas. 
We say that $X$ is \textit{$L$-consistent} if $L \nvdash \bigwedge X \to \bot$ where $\bigwedge X$ is a conjunction of all elements of $X$. 
We say that $X$ is \textit{$A$-maximally $L$-consistent} if $X \subseteq \num{\Sub(A)}$, $X$ is $L$-consistent, and for any $B \in \num{\Sub(A)}$, either $B \in X$ or ${\sim} B \in X$. 
It is easily shown that for each $L$-consistent subset $X$ of $\num{\Sub(A)}$, there exists an $A$-maximally $L$-consistent superset of $X$. 

\begin{thm}\label{FFP}
Let $L \in \{\MN, \MNF, \MNP, \MND, \MNPF, \MNDF\}$. 
Then, for any $\LB$-formula $A$, the following are equivalent:
\begin{enumerate}
	\item $L \vdash A$. 
	\item $A$ is valid in all $L$-frames. 
	\item $A$ is valid in all finite $L$-frames. 
\end{enumerate}
\end{thm}
\begin{proof}
The implications $(1 \Rightarrow 2)$ and $(2 \Rightarrow 3)$ are straightforward by Proposition \ref{FC}. 
We show the contrapositive of the implication $(3 \Rightarrow 1)$. 
Suppose $L \nvdash A$. 
Then, $\{{\sim} A\}$ is $L$-consistent, and thus there exists an $A$-maximally $L$-consistent set $X_A$ containing ${\sim} A$. 
Let $W$ be the set of all $A$-maximally $L$-consistent sets. 
Since $\num{\Sub(A)}$ is a finite set, $W$ is also a finite set containing $X_A$. 
We define a binary relation $\prec_L$ on $W$ depending on $L$ as follows: Let $x \in W$ and $V \in \mathcal{P}(W)$.  
\begin{itemize}
	\item For $L \in \{\MN, \MNP, \MND\}$, 
\[
	x \prec_L V : \iff V \neq \emptyset\ \&\ (\forall \Box B \in \num{\Sub(A)}) \bigl(\Box B \in x \Rightarrow (\exists y \in V) (B \in y) \bigr). 
\]
	\item For $L \in \{\MNF, \MNPF, \MNDF\}$, 
\begin{align*}
	x \prec_L V : & \iff V \neq \emptyset\\
	& \& \ (\forall \Box B \in \num{\Sub(A)}) \bigl(\Box B \in x \Rightarrow (\exists y_0, y_1 \in V) (B \in y_0\ \&\ \Box B \in y_1) \bigr). 
\end{align*}
\end{itemize}
Since $\prec_L$ satisfies the condition of monotonicity, $(W, \prec_L)$ is a finite $\MN$-frame. 
We define a satisfaction relation $\Vdash$ on $(W, \prec_L)$ as follows: For each $x \in W$ and propositional variable $p$, 
\[
	 x \Vdash p : \iff p \in x. 
\]

\begin{cl}\label{TL}
For any $B \in \num{\Sub(A)}$ and $x \in W$, 
\[
	x \Vdash B \iff B \in x.
\] 
\end{cl}
\begin{proof}
This is proved by induction on the construction of $B$. 
We give only a proof of the case that $B$ is of the form $\Box C$. 

$(\Rightarrow)$: 
Suppose $\Box C \notin x$.
Since $x$ is $A$-maximally $L$-consistent, $\neg \Box C \in x$. 
We would like to show that there exists a $V \in \mathcal{P}(W)$ such that $x \prec_L V$ and $y \nVdash C$ for all $y \in V$. 
Suppose, towards a contradiction, that $\{D, {\sim}C\}$ is $L$-inconsistent for some $D$ with $\Box D \in x$. 
Then, $L \vdash D \to C$. 
By the rule \textsc{RM}, $L \vdash \Box D \to \Box C$, and hence $L \vdash \Box D \land \neg \Box C \to \bot$. 
This contradicts the $L$-consistency of $x$. 
Thus, for any $\Box D \in x$, the set $\{D, {\sim}C\}$ is $L$-consistent, and hence we find a $y_D \in W$ such that $\{D, {\sim}C\} \subseteq y_D$. 

We distinguish the following two cases: 
\begin{itemize}
	\item Case 1: $L \in \{\MN, \MNP, \MND\}$. \\
Let $V : = \{y_D \mid \Box D \in x\}$. 
Since $\Box \top \in x$, $y_\top \in V$, and hence $V \neq \emptyset$. 
Then, by the definition of $\prec_L$, we obtain $x \prec_L V$. 
For each $y_D \in V$, we have ${\sim} C \in y_D$, and hence $C \notin y_D$. 
By the induction hypothesis, $y_D \nVdash C$. 

	\item Case 2: $L \in \{\MNF, \MNPF, \MNDF\}$. \\
Suppose, towards a contradiction, that $\{\Box D, {\sim}C\}$ is $L$-inconsistent for some $D$ with $\Box D \in x$. 
Then, $L \vdash \Box D \to C$. 
By the rule \textsc{RM}, $L \vdash \Box \Box D \to \Box C$. 
Since $L \vdash \Box D \to \Box \Box D$, we obtain $L \vdash \Box D \land \neg \Box C \to \bot$, a contradiction. 
Thus, for any $\Box D \in x$, the set $\{\Box D, {\sim}C\}$ is $L$-consistent, and hence we have a $z_D \in W$ such that $\{\Box D, {\sim}C\} \subseteq z_D$. 

Let $V : = \{y_D, z_D \mid \Box D \in x\}$. 
As above, $V \neq \emptyset$. 
By the definition of $\prec_L$, we obtain $x \prec_L V$. 
For each $w \in V$, we have ${\sim} C \in w$, and hence $C \notin w$. 
By the induction hypothesis, $w \nVdash C$. 
\end{itemize}

In either case, we conclude $x \nVdash \Box C$.

$(\Leftarrow)$: 
Suppose $\Box C \in x$. 
Let $V \in \mathcal{P}(W)$ be such that $x \prec_L V$. 
Then, by the definition of $\prec_L$, there exists a $y \in V$ such that $C \in y$. 
By the induction hypothesis, $y \Vdash C$. 
Hence, $x \Vdash \Box C$. 
\end{proof}

Since ${\sim} A \in X_A$, we have $A \notin X_A$. 
By Claim \ref{TL}, we conclude $X_A \nVdash A$. 
We have finished our proof of the case that $L$ is $\MN$. 

In the case of $L \in \{\MNF, \MNPF, \MNDF\}$, we prove that $(W, \prec_L)$ is transitive. 
Suppose that $x \prec_L V$ and $y \prec_L U_y$ for all $y \in V$. 
Let $\Box B \in x$. 
By the definition of $\prec_L$, there exists a $y_1 \in V$ such that $\Box B \in y_1$. 
Also by the definition of $\prec_L$, there are $z_0, z_1 \in U_{y_1}$ such that $B \in z_0$ and $\Box B \in z_1$. 
Since $z_0, z_1 \in \bigcup_{y \in V} U_y$, we obtain $x \prec_L \bigcup_{y \in V} U_y$. 

In the case of $L \in \{\MNP, \MNPF\}$, we show that $(W, \prec_L)$ is an $\MNP$-frame. 
Since every $A$-maximally $L$-consistent set contains the formula $\neg \Box \bot$, by Claim \ref{TL}, $\neg \Box \bot$ is valid in $(W, \prec_L)$. 
By Proposition \ref{FC}, we have that $(W, \prec_L)$ is an $\MNP$-frame. 

In the case of $L = \MND$, we show that $(W, \prec_L)$ is an $\MND$-frame. 
Let $x \in W$ and $V \in \mathcal{P}(W)$ be such that $\neg (x \prec_L V)$, and we would like to show $x \prec_L (W \setminus V)$. 
Let $\Box B \in x$. 
By the definition of $\prec_L$, there exists a $\Box C \in x$ such that for any $z \in V$, $C \notin z$. 
Suppose, towards a contradiction, that the set $\{B, C\}$ is $\MND$-inconsistent. 
Then, $\MND \vdash B \to \neg C$. 
By the rule \textsc{RM}, $\MND \vdash \Box B \to \Box \neg C$. 
Since $\MND \vdash \Box \neg C \to \neg \Box C$, we obtain $\MND \vdash \Box B \to \neg \Box C$, and hence $\MND \vdash \Box B \land \Box C \to \bot$. 
This contradicts the $\MND$-consistency of $x$.
Therefore, the set $\{B, C\}$ is $\MND$-consistent, and there exists a $y \in W$ such that $\{B, C\} \subseteq y$. 
Since $C \in y$, we have $y \notin V$, and hence $y \in W \setminus V$. 
We have shown that for any $\Box B \in x$, there exists a $y \in W \setminus V$ such that $B \in y$. 
By the definition of $\prec_L$, we conclude $x \prec_L (W \setminus V)$. 
Therefore, $(W, \prec_L)$ is an $\MND$-frame. 

In the case of $L = \MNDF$, we also show that $(W, \prec_L)$ is an $\MND$-frame. 
Let $x \in W$ and $V \in \mathcal{P}(W)$ be such that $\neg (x \prec_L V)$, and we would like to show $x \prec_L (W \setminus V)$. 
Let $\Box B \in x$. 
By the definition of $\prec_L$, there exists a $\Box C \in x$ such that (i) for any $z \in V$, $C \notin z$ or (ii) for any $z \in V$, $\Box C \notin z$. 
In the case (i), as above, it is shown that the set $\{B, C\}$ is $\MNDF$-consistent. 
By using the fact $\MNDF \vdash \Box B \to \Box \Box B$, it is shown that the set $\{\Box B, C\}$ is also $\MNDF$-consistent. 
Thus, there are $y, z \in W$ such that $\{B, C\} \subseteq y$ and $\{\Box B, C\} \subseteq z$. 
Since $C$ is in $y$ and $z$, they are not in $V$. 
Hence, $y, z \in W \setminus V$. 
In the case (ii), it is also shown that $\{B, \Box C\}$ and $\{\Box B, \Box C\}$ are $\MNDF$-consistent by using the fact $\MNDF \vdash \neg \Box \Box C \to \neg \Box C$. 
Thus, we obtain that there are $y, z \in W \setminus V$ such that $B \in y$ and $\Box B \in z$. 
In either case, we conclude $x \prec_L (W \setminus V)$. 
\end{proof}

From our proof of Theorem \ref{FFP}, we obtain the following corollary. 

\begin{cor}\label{PR}
For each $L \in \{\MN, \MNF, \MNP, \MND, \MNPF, \MNDF\}$, there exists a primitive recursive decision procedure for provability in $L$. 
\end{cor}

\section{Arithmetical completeness of $\MN$ and $\MNF$}\label{Sec:MN}

As mentioned before, the rule \textsc{RM} is a modal counterpart of the derivability condition $\M$. 
More precisely, the following proposition is easily proved. 

\begin{prop}[The arithmetical soundness of $\MN$ and $\MNF$]
Let $\PR_T(x)$ be a provability predicate of $T$ satisfying $\M$. 
Then, $\MN \subseteq \PL(\PR_T)$ holds. 
Furthermore, $\PR_T(x)$ satisfies the condition $\D{3}$ if and only if $\MNF \subseteq \PL(\PR_T)$. 
\end{prop}

In this section, we prove the arithmetical completeness theorems of the logics $\MN$ and $\MNF$. 
Before proving the theorems, we prepare several notions. 
An $\LA$-formula is called \textit{propositionally atomic} if it is either atomic or of the from $Q x \varphi$ for $Q \in \{\forall, \exists\}$. 
Notice that every $\LA$-formula is a Boolean combination of propositionally atomic formulas. 
For each propositionally atomic formula $\varphi$, we prepare a propositional variable $p_\varphi$. 
We define a primitive recursive injection $I$ from $\LA$-formulas into propositional formulas recursively as follows: 
\begin{itemize}
	\item $I(\varphi)$ is $p_\varphi$ for every propositionally atomic formula $\varphi$, 
	\item $I(\neg \varphi)$ is $\neg I(\varphi)$, 
	\item $I(\varphi \circ \psi)$ is $I(\varphi) \circ I(\psi)$ for $\circ \in \{\land, \lor, \to\}$. 
\end{itemize}
We say that an $\LA$-formula $\varphi$ is a \textit{tautological consequence} (\textit{t.c.}) of a finite set $X$ of $\LA$-formulas if $I\bigl(\bigwedge X \to \varphi \bigr)$ is a tautology. 

For each $m \in \omega$, let $F_m$ be the set of all $\LA$-formulas, whose G\"{o}del numbers are less than or equal to $m$ and let $P_{T, m}$ be the finite set 
\[
	\{ \varphi \in F_m \mid \N \models \exists y \leq \num{m} \ \Proof_{T}(\gn{\varphi}, y) \}.
\] 
Then, it is shown that for each $m \in \omega$ and $\varphi$, whether $\varphi$ is a t.c.~of $P_{T, m}$ is primitive recursively determined. 

Next, for each $m \in \omega$, we define a binary relation $\tto_{m}$ between $\LA$-formulas as follows: 
\begin{center}
	$\varphi \tto_m \rho$ if and only if there exists a finite sequence of $\LA$-formulas $\psi_0, \ldots, \psi_k$ such that $\psi_0 \equiv \varphi$, $\psi_k \equiv \rho$, and $\psi_i \to \psi_{i+1} \in P_{T,m}$ for each $i < k$.
\end{center}
Here, $\psi \equiv \varphi$ denotes that $\psi$ and $\varphi$ are identical as sequences of symbols. 
It can be easily shown that the ternary relation $\{(\varphi, \rho, m) \mid \varphi \tto_m \rho\}$ is primitive recursive. 
The following proposition is also straightforward. 

\begin{prop}\label{FProp}
Let $m \in \omega$ and let $\varphi$ and $\rho$ be any $\LA$-formulas. 
\begin{enumerate}
	\item If $\varphi$ is a t.c.~of $P_{T,m}$, then $\varphi$ is provable in $T$.
	\item If $\varphi \tto_m \rho$ holds, then $\varphi \to \rho$ is a t.c.~of $P_{T, m}$. 
	\item The binary relation $\tto_m$ is transitive and reflexive. 
\end{enumerate}
\end{prop}

We are ready to prove the following uniform version of the arithmetical completeness of $\MN$ and $\MNF$. 
For the uniform arithmetical completeness theorem of $\GL$, see \cite[p.~132]{Boo}.

\begin{thm}[The uniform arithmetical completeness of $\MN$ and $\MNF$]\label{UACTMN}
For $L \in \{\MN, \MNF\}$, there exists a $\Sigma_1$ provability predicate $\PR_T(x)$ of $T$ satisfying $\M$ such that
\begin{enumerate}
	\item for any $\LB$-formula $A$ and any arithmetical interpretation $f$ based on $\PR_T(x)$, if $L \vdash A$, then $\PA \vdash f(A)$, and
	\item there exists an arithmetical interpretation $f$ based on $\PR_T(x)$ such that for any $\LB$-formula $A$, $L \vdash A$ if and only if $T \vdash f(A)$. 
\end{enumerate}
\end{thm}
\begin{proof}
Let $L \in \{\MN, \MNF\}$. 
By Corollary \ref{PR}, we have a primitive recursive procedure to determine whether a given $\LB$-formula is $L$-provable or not.
Then, let $\langle A_n \rangle_{n \in \omega}$ be a primitive recursive enumeration of all $L$-unprovable formulas. 
For each $A_n$, we can primitive recursively construct a finite $L$-model $(W_n, \prec_n, \Vdash_n)$ falsifying $A_n$. 
We may assume that the sets $\{W_n\}_{n \in \omega}$ are pairwise disjoint and $\bigcup_{n \in \omega} W_n = \omega \setminus \{0\}$. 
We may also assume that $\langle (W_n, \prec_n, \Vdash_n) \rangle_{n \in \omega}$ is primitive recursively represented in $\PA$, and several basic properties of this enumeration is provable in $\PA$. 
For each $n \in \omega$, a function $c: \mathcal{P}(W_n) \rightarrow W_n$ is called an \textit{$n$-choice function} if for each $V \in \mathcal{P}(W_n)$, $c(V) \in V$.

We simultaneously define two primitive recursive functions $h_0$ and $g_0$ by using the double recursion theorem. 
Firstly, we define the function $h_0$. 
In the definition of $h_0$, the formula $\PR_{g_0}(x)$ defined as $\exists y (x = g_0(y) \land \Fml(x))$ based on $g_0$ is used. 

\begin{itemize}
	\item $h_0(0) = 0$. 
	\item $h_0(m+1) = \begin{cases}
				i & \text{if}\ h_0(m) = 0\\
					& \&\ i = \min \bigl\{j \in \omega \setminus \{0\} \mid \neg S_0(\num{j}) \ \text{is a t.c.~of}\ P_{T, m}\\
					& \quad\ \text{or}\ \exists \varphi \bigl[S_0(\num{j}) \to \neg \PR_{g_0}(\gn{\varphi})\ \text{is a t.c.~of}\ P_{T, m}\\
					& \quad \quad \quad \& \ \neg \varphi \in F_m\ \&\ \neg \varphi \notin X_m \cup Y_{j, m} \bigr]\bigr\}\\
			h_0(m) & \text{otherwise}.
		\end{cases}$
\end{itemize}
Here, $S_0(x)$ is the $\Sigma_1$ formula $\exists y(h_0(y) = x)$. 
Also, for each $j \in W_n$ and number $m$, $X_m$ and $Y_{j, m}$ are sets defined as follows: 
\begin{itemize}
	\item $X_m : = \{\varphi \in F_m \mid \exists \psi \in P_{T, m}$ s.t.~$\psi \tto_m \varphi\}$. 
	\item $Y_{j, m} : = \{\varphi \in F_m \mid \exists V \in \mathcal{P}(W_n)$ s.t.~$j \prec_n V$ and $\forall l \in V\, (S_0(\num{l}) \tto_m \varphi)\}$. 
\end{itemize}

We have to show that the function $h_0$ is actually primitive recursive. 
For this purpose, it suffices to show that the value of $h_0(m+1)$ is bounded by some number that is primitive recursively computed from $m$. 
In particular, it suffices to find a primitive recursively computed bound of $h_0(m+1)$ when $h_0(m) = 0$ and $h_0(m+1) \neq 0$. 
Here, we may assume that our enumeration $\langle (W_n, \prec_n, \Vdash_n) \rangle_{n \in \omega}$ is coded so that for each $m \in \omega$, numbers $n$ and $j_0 \in W_n$ satisfying the following condition are computed in a primitive recursive way: 
\begin{itemize}
	\item for any $V \in \mathcal{P}(W_n)$ with $j_0 \prec_n V$, there exists an $l \in V$ such that $l \geq m$.
\end{itemize} 
Then, we claim that if $h_0(m) = 0$ and $h_0(m+1) = i \neq 0$, then $i \leq \max \{j_0, m\}$. 
And this claim guarantees the primitive recursiveness of $h_0$. 

We show the claim. 
Suppose $h_0(m) = 0$ and $h_0(m+1) = i \neq 0$. 
If $P_{T, m}$ is not propositionally satisfiable, then $\neg S_0(\num{1})$ is a t.c.~of $P_{T, m}$, and hence $i = 1 \leq \max \{j_0, m\}$. 
So we assume that $P_{T, m}$ is propositionally satisfiable. 
In the case that $\neg S_0(\num{i})$ is a t.c.~of $P_{T, m}$, then $S_0(\num{i})$ is a subformula of a formula contained in $P_{T, m}$ because $S_0(\num{i})$ is propositionally atomic, and thus $i < m$. 
In the case that there exists a formula $\varphi$ such that $S_0(\num{i}) \to \neg \PR_{g_0}(\gn{\varphi})$ is a t.c.~of $P_{T, m}$, $\neg \varphi \in F_m$, and $\neg \varphi \notin X_m \cup Y_{j, m}$, then we distinguish the two cases based on whether $\neg \PR_{g_0}(\gn{\varphi})$ is a t.c.~of $P_{T, m}$. 
If $\neg \PR_{g_0}(\gn{\varphi})$ is a t.c.~of $P_{T, m}$, then $S_0(\num{j_0}) \to \neg \PR_{g_0}(\gn{\varphi})$ is also a t.c.~of $P_{T, m}$. 
By the choice of $j_0$, for any $V \in \mathcal{P}(W_n)$ with $j_0 \prec_n V$, there exists an $l \in V$ such that $S_0(\num{l}) \notin F_m$, and hence $S_0(\num{l}) \not \tto_m \neg \varphi$. 
Thus, $\neg \varphi \notin X_m \cup Y_{j_0, m}$. 
It follows that $i \leq j_0$. 
If $\neg \PR_{g_0}(\gn{\varphi})$ is not a t.c.~of $P_{T, m}$, then $S_0(\num{i})$ is a subformula of a formula contained in $P_{T, m}$ because $S_0(\num{i}) \to \neg \PR_{g_0}(\gn{\varphi})$ is a t.c.~of $P_{T, m}$ and $S_0(\num{i})$ is propositionally atomic. 
It follows that $i < m$. 

Secondly, we define a primitive recursive function $g_0$ step by step, that enumerates all theorems of $T$. 
The definition of $g_0$ consists of Procedures 1 and 2, and it starts with Procedure 1. 
The values of $g_0(0), g_0(1), \ldots$ are defined by referring to $T$-proofs and the values of the function $h_0$. 
At the first time $h_0(m + 1) \neq 0$, the definition of $g_0$ is switched to Procedure 2. 
In the definition of $g_0$, we identify each formula with its G\"{o}del number.

\vspace{0.1in}

\textsc{Procedure 1}. 

Stage $m$:
\begin{itemize}
 \item If $h_0(m+1) = 0$,
\begin{equation*}
g_0(m)  = \begin{cases}
     \varphi & \text{if}\ m\ \text{is a}\ T\text{-proof of}\ \varphi. \\
             0 & \text{if}\ m\ \text{is not a}\ T\text{-proof of any formula}.
           \end{cases}
\end{equation*} 

Then, go to Stage $m+1$. 

\item If $h_0(m+1) \neq 0$, go to Procedure 2.
\end{itemize}

\textsc{Procedure 2}.

Suppose $m$ and $i \neq 0$ satisfy $h_0(m)=0$ and $h_0(m+1)=i$. 
Let $n$ be a number such that $i \in W_n$. 
Let $\langle \xi_t \rangle_{t \in \omega}$ be the repetition-free primitive recursive enumeration of all $\LA$-formulas in ascending order of G\"odel numbering, which is introduced in Section \ref{Sec:PP}. 
The values of $g_0(m), g_0(m+1), \ldots$ are defined depending on whether $i \Vdash_n \Box \bot$. 

If $i \Vdash_n \Box \bot$, then for any $t$, we define
\[
	g_0(m+t)= \xi_t. 
\]

If $i \nVdash_n \Box \bot$, then for any $t$, we define
\begin{equation*}
g_0(m+t)= \begin{cases} \xi_t & \text{if}\ \neg \xi_t \notin X_{m-1} \cup Y_{i, m-1}, \\
				0 & \text{otherwise}.
		\end{cases}
\end{equation*}

The definition of $g_0$ has just been completed. 
We define the formulas $\Prf_{g_0}(x,y)$ and $\PR_{g_0}(x)$ as follows: 
\begin{itemize}
	\item $\Prf_{g_0}(x, y) \equiv (x=g_0(y) \wedge \Fml(x))$, 
	\item $\PR_{g_0}(x) \equiv \exists y \Prf_{g_0}(x, y)$. 
\end{itemize}

\begin{cl}\label{CLh}
\leavevmode
\begin{enumerate}
	\item $\PA \vdash \forall x \forall y(0 < x < y \land S_0(x) \to \neg S_0(y))$. 
	\item $\PA \vdash \Prov_T(\gn{0=1}) \leftrightarrow \exists x(S_0(x) \land x \neq 0)$. 
	\item For each $i \in \omega \setminus \{0\}$, $T \nvdash \neg S_0(\num{i})$. 
	\item For each $m \in \omega$, $\PA \vdash \forall x \forall y(h_0(x) = 0 \land h_0(x+1) = y \land y \neq 0 \to x > \num{m})$.
\end{enumerate}
\end{cl}
\begin{proof}
1. This is clear because $\PA \vdash \forall x \forall u \forall v(x \neq 0 \land h_0(u) = x \land v \geq u \to h_0(v) = x)$. 

2. We argue in $\PA$. 
$(\to)$: If $T$ is inconsistent, then $P_{T, m}$ is propositionally unsatisfiable for some $m$. 
Hence, $\neg S_0(\num{1})$ is a t.c.~of $P_{T, m}$, and thus $h_0(m+1) \neq 0$.  
This means that $S_0(i)$ holds for some $i \neq 0$. 

$(\leftarrow)$: 
Suppose that $S_0(i)$ holds for some $i \neq 0$. 
Let $m$ and $n$ be such that $h_0(m) = 0$, $h_0(m+1) = i$, and $i \in W_n$. 
We would like to show that $T$ is inconsistent. 
By the definition of $h_0$, we distinguish the following two cases: 

Case 1: $\neg S_0(\num{i})$ is a t.c.~of $P_{T, m}$. \\
	Then, $\neg S_0(\num{i})$ is $T$-provable. 
	Since $S_0(\num{i})$ is a true $\Sigma_1$ sentence, it is provable in $T$. 
	Therefore, $T$ is inconsistent. 

Case 2: There exists a $\varphi$ such that $S_0(\num{i}) \to \neg \PR_{g_0}(\gn{\varphi})$ is a t.c.~of $P_{T, m}$, $\neg \varphi \in F_m$, and $\neg \varphi \notin X_m \cup Y_{i, m}$. \\
	Since $\neg \varphi \notin X_{m-1} \cup Y_{i, m-1}$, $g_0(m + t) = \varphi$ for some $t$.  
	Then, $S_0(\num{i}) \land \PR_{g_0}(\gn{\varphi})$ is a true $\Sigma_1$ sentence, and so it is provable in $T$.
	On the other hand, since $S_0(\num{i}) \to \neg \PR_{g_0}(\gn{\varphi})$ is a t.c.~of $P_{T, m}$, it is also provable in $T$. 
	Therefore, $T$ is inconsistent. 

3. Suppose $T \vdash \neg S_0(\num{i})$ for $i \neq 0$. 
Let $p$ be a $T$-proof of $\neg S_0(\num{i})$. 
Then, $\neg S_0(\num{i}) \in P_{T, p}$, and thus $h_0(p+1) \neq 0$. 
That is, $\exists x(S_0(x) \land x \neq 0)$ is true. 
By clause 2, $T$ is inconsistent, a contradiction. 

4. This is because $\PA \vdash h_0(\num{m}) = 0$ for each $m \in \omega$. 
\end{proof}

\begin{cl}\label{CL1}\leavevmode
\begin{enumerate}
\item $\PA+ \neg \Prov_T(\gn{0=1}) \vdash \forall x \bigl(\Prov_{T}(x) \leftrightarrow \PR_{g_0}(x) \bigr)$.
\item For any $n \in \omega$ and any $\LA$-formula $\varphi$, $\PA \vdash \Proof_{T}(\gn{\varphi},\num{n}) \leftrightarrow \Prf_{g_0}(\gn{\varphi},\num{n})$.
\end{enumerate}
\end{cl}
\begin{proof}
1. By the definition of $g_0$, 
\[
	\PA \vdash \neg \exists x(S_0(x) \wedge x \neq 0) \rightarrow \forall x \bigl(\Prov_{T}(x)\leftrightarrow \PR_{g_0}(x) \bigr)
\]
holds.   
Since $\Prov_T(\gn{0=1})$ and $\exists x (S_0(x) \land x \neq 0)$ are equivalent in $\PA$ by Claim \ref{CLh}.2, we obtain 
\[
	\PA \vdash \neg \Prov_T(\gn{0=1}) \rightarrow \forall x \bigl(\Prov_{T}(x) \leftrightarrow \PR_{g_0}(x) \bigr).
\]

2. Let $n$ be a natural number. 
We reason in $\PA$: 
By Claim \ref{CLh}.4, the construction of $g_0$ does not switch to Procedure 2 before Stage $n+1$. 
Then, $n$ is a $T$-proof of $\varphi$ if and only if $g_0(n) = \varphi$. 
\end{proof}

Therefore, our formula $\PR_{g_0}(x)$ is a $\Sigma_1$ provability predicate of $T$.

\begin{cl}\label{CL2}
$\PA$ proves the following statement: 
``Let $m$, $i \neq 0$, and $n$ be such that $h_0(m) = 0$, $h_0(m + 1) = i$, $i \in W_n$ and $i \nVdash_n \Box \bot$. 
Then, for any $\LA$-formula $\varphi$, $\neg \varphi \in X_{m-1} \cup Y_{i, m-1}$ if and only if $\neg \PR_{g_0}(\gn{\varphi})$ holds''. 
\end{cl}
\begin{proof}
We argue in $\PA$: Let $m$, $i \neq 0$ and $n$ satisfy $h_0(m)=0$, $h_0(m+1)=i$, $i \in W_n$, and $i \nVdash_n \Box \bot$. 

$(\Rightarrow)$: Assume $\neg \varphi \in X_{m-1} \cup Y_{i, m-1}$. 
Suppose, towards a contradiction, that $\varphi$ is output by $g_0$. 
Since $\neg \varphi \in X_{m-1} \cup Y_{i, m-1}$, $\varphi$ is output in Procedure 1. 
Then, there exists a $T$-proof $p \leq m-1$ of $\varphi$, and hence $\varphi \in P_{T, m-1}$. 
We distinguish the following two cases:

Case 1: $\neg \varphi \in X_{m-1}$. \\
		Then, there exists a $\psi \in P_{T, m-1}$ such that $\psi \tto_{m-1} \neg \varphi$. 
		By Proposition \ref{FProp}.2, $\psi \to \neg \varphi$ is a t.c.~of $P_{T, m-1}$. 
		Since $\varphi \in P_{T, m-1}$, $P_{T, m-1}$ is propositionally unsatisfiable. 
		In particular, $\neg S_0(\num{1})$ is a t.c.~of $P_{T, m -1}$. 
		By the definition of $h_0$, $h_0(m) = 1 \neq 0$. 
		This is a contradiction. 

Case 2: $\neg \varphi \in Y_{i, m-1}$. \\
		Then, there exists a $V \in \mathcal{P}(W_n)$ such that $i \prec_n V$ and $S_0(\num{j}) \tto_{m-1} \neg \varphi$ for any $j \in V$. 
		Since $V$ is non-empty, there exists a $j \neq 0$ such that $S_0(\num{j}) \tto_{m-1} \neg \varphi$. 
		By Proposition \ref{FProp}.2, $S_0(\num{j}) \to \neg \varphi$ is a t.c.~of $P_{T, m-1}$. 
		Since $\varphi \in P_{T, m-1}$, $\neg S_0(\num{j})$ is a t.c.~of $P_{T, m-1}$. 
		Hence, $h_0(m) \neq 0$, a contradiction. 

Therefore, $\varphi$ is not output by $g_0$. 
In other words, $\neg \PR_{g_0}(\gn{\varphi})$ holds. 

$(\Leftarrow)$: This implication directly follows from the definition of $g_0$. 
\end{proof}

We prove that $\PR_{g_0}(x)$ satisfies the condition $\M$.

\begin{cl}\label{CL3}
For any $\LA$-formulas $\varphi$ and $\psi$, if $T \vdash \varphi \to \psi$, then $\PA \vdash \PR_{g_0}(\gn{\varphi}) \to \PR_{g_0}(\gn{\psi})$. 
\end{cl}
\begin{proof}
Suppose $T \vdash \varphi \to \psi$. 
Then, $\PA \vdash \Prov_T(\gn{\varphi}) \to \Prov_T(\gn{\psi})$. 
Since $\PA + \neg \exists x (S_0(x)\wedge x\neq 0) \vdash \forall x \bigl(\Prov_T(x) \leftrightarrow \PR_{g_0}(x) \bigr)$, we have that $\PA + \neg \exists x (S_0(x)\wedge x\neq 0)$ proves $\PR_{g_0}(\gn{\varphi}) \to \PR_{g_0}(\gn{\psi})$. 

Next, we show $\PA + \exists x (S_0(x)\wedge x\neq 0) \vdash \PR_{g_0}(\gn{\varphi}) \to \PR_{g_0}(\gn{\psi})$. 
By the supposition, $\neg \psi \to \neg \varphi$ has a standard proof $p$. 
We reason in the theory $\PA + \exists x (S_0(x)\wedge x\neq 0)$: 
Let $m$, $i \neq 0$, and $n$ satisfy $h_0(m)=0$, $h_0(m+1)=i$, and $i \in W_n$. 
If $i \Vdash_n \Box \bot$, then $g_0$ outputs all $\LA$-formulas, and hence $\PR_{g_0}(\gn{\psi})$ holds, and hence $\PR_{g_0}(\gn{\varphi}) \to \PR_{g_0}(\gn{\psi})$ holds. 

So, we may assume that $i \nVdash_n \Box \bot$. 
Suppose that $\neg \PR_{g_0}(\gn{\psi})$ holds. 
By Claim \ref{CL2}, $\neg \psi \in X_{m-1} \cup Y_{i, m-1}$. 
We distinguish the following two cases: 

Case 1: $\neg \psi \in X_{m-1}$. \\
Then, there exists some $\rho \in P_{T,m-1}$ such that $\rho \tto_{m-1} \neg \psi$. 
By Claim \ref{CLh}.4, $m > p$ and hence $\neg \psi \to \neg \varphi \in P_{T, m-1}$. 
Then, we obtain $\rho \tto_{m-1} \neg \varphi$.
Hence, $\neg \varphi \in X_{m-1}$. 

Case 2: $\neg \psi \in Y_{i, m-1}$. \\
Then, there exists a $V \in \mathcal{P}(W_n)$ such that $i \prec_n V$ and $S_0(\num{j}) \tto_{m-1} \neg \psi$ for any $j \in V$. 
Since $\neg \psi \to \neg \varphi \in P_{T, m-1}$, we obtain $S_0(\num{j})\tto_{m-1} \neg \varphi$ for any $j \in V$. 
Thus, $\neg \varphi \in Y_{i, m-1}$. 

In either case, we have $\neg \varphi \in X_{m-1} \cup Y_{i, m-1}$. 
By Claim \ref{CL2} again, we conclude that $\neg \PR_{g_0}(\gn{\varphi})$ holds.

We have proved $\PA + \exists x (S_0(x)\wedge x\neq 0) \vdash \PR_{g_0}(\gn{\varphi}) \to \PR_{g_0}(\gn{\psi})$. 
Finally, by the law of excluded middle, we conclude $\PA \vdash \PR_{g_0}(\gn{\varphi}) \to \PR_{g_0}(\gn{\psi})$. 
\end{proof}

We prove that when $L = \MNF$, $\PR_{g_0}(x)$ satisfies the condition $\D{3}$. 

\begin{cl}\label{CLF}
If $L = \MNF$, then for any $\LA$-formula $\varphi$, $\PA \vdash \PR_{g_0}(\gn{\varphi}) \to \PR_{g_0}(\gn{\PR_{g_0}(\gn{\varphi})})$. 
\end{cl}
\begin{proof}
Since $\PR_{g_0}(\gn{\varphi})$ is a $\Sigma_1$ sentence, $\PA \vdash \PR_{g_0}(\gn{\varphi}) \to \Prov_T(\gn{\PR_{g_0}(\gn{\varphi})})$ holds. 
Then, we have $\PA + \neg \Prov_T(\gn{0=1}) \vdash \PR_{g_0}(\gn{\varphi}) \to \PR_{g_0}(\gn{\PR_{g_0}(\gn{\varphi})})$ by Claim \ref{CL1}.1. 
By Claim \ref{CLh}.2, $\PA + \neg \exists x (S_0(x) \land x \neq 0)$ proves $\PR_{g_0}(\gn{\varphi}) \to \PR_{g_0}(\gn{\PR_{g_0}(\gn{\varphi})})$. 

On the other hand, we show that $\PA + \exists x (S_0(x) \land x \neq 0)$ also proves $\PR_{g_0}(\gn{\varphi}) \to \PR_{g_0}(\gn{\PR_{g_0}(\gn{\varphi})})$. 
We reason in $\PA + \exists x (S_0(x) \land x \neq 0)$: 
Let $m$, $i \neq 0$, and $n$ be such that $h_0(m) = 0$, $h_0(m+1) = i$, and $i \in W_n$. 
If $i \Vdash_n \Box \bot$, then $g_0$ outputs all $\LA$-formulas, and hence $\PR_{g_0}(\gn{\varphi}) \to \PR_{g_0}(\gn{\PR_{g_0}(\gn{\varphi})})$ trivially holds. 
Then, we may assume $i \nVdash_n \Box \bot$. 
Suppose that $\neg \PR_{g_0}(\gn{\PR_{g_0}(\gn{\varphi})})$ holds. 
By Claim \ref{CL2}, $\neg \PR_{g_0}(\gn{\varphi}) \in X_{m-1} \cup Y_{i, m-1}$. 
Assume, towards a contradiction, $\neg \varphi \notin X_{m-1} \cup Y_{i, m-1}$. 
We distinguish the following cases. 

Case 1: $\neg \PR_{g_0}(\gn{\varphi}) \in X_{m-1}$. \\
Then, there exists a $\psi \in P_{T, m-1}$ such that $\psi \tto_{m-1} \neg \PR_{g_0}(\gn{\varphi})$. 
Hence, $\neg \PR_{g_0}(\gn{\varphi})$ is a t.c.~of $P_{T, m-1}$. 
Then, $\neg S_0(\num{i}) \to \neg \PR_{g_0}(\gn{\varphi})$ is also a t.c.~of $P_{T, m-1}$. 
Since $\neg \varphi \notin X_{m-1} \cup Y_{i, m-1}$ and $\neg \varphi \in F_{m-1}$, we have $h_0(m) \neq 0$. 
This is a contradiction. 

Case 2: $\neg \PR_{g_0}(\gn{\varphi}) \in Y_{i, m-1}$. \\
Then, there exists a $V \in \mathcal{P}(W_n)$ such that $S_0(\num{j}) \tto_{m-1} \neg \PR_{g_0}(\gn{\varphi})$ for any $j \in V$. 
Suppose, towards a contradiction, that $\neg \varphi \in Y_{j, m-1}$ for all $j \in V$. 
Then, for each $j \in V$, there exists a $U_j \in \mathcal{P}(W_n)$ such that $j \prec_n U_j$ and $S_0(\num{l}) \tto_{m-1} \neg \varphi$ for any $l \in U_j$. 
Let $U : = \bigcup_{j \in V} U_j$. 
Since $(W_n, \prec_n)$ is transitive, we have $i \prec_n U$. 
Also, $S_0(\num{l}) \tto_{m-1} \neg \varphi$ for any $l \in U$. 
It follows that $\neg \varphi \in Y_{i, m-1}$, a contradiction. 
Therefore, we obtain that $\neg \varphi \notin Y_{j, m-1}$ for some $j \in V$. 
Then, $h_0(m) \neq 0$ because $S_0(\num{j}) \to \neg \PR_{g_0}(\gn{\varphi})$ is a t.c.~of $P_{T, m-1}$, $\neg \varphi \in F_{m-1}$, and $\neg \varphi \notin X_{m-1} \cup Y_{j, m-1}$. 
This is a contradiction. 

Therefore, $\neg \varphi \in X_{m-1} \cup Y_{i, m-1}$. 
By Claim \ref{CL2} again, we obtain that $\neg \PR_{g_0}(\gn{\varphi})$ holds. 

At last, by the law of excluded middle, we conclude $\PA \vdash \PR_{g_0}(\gn{\varphi}) \to \PR_{g_0}(\gn{\PR_{g_0}(\gn{\varphi})})$. 
\end{proof}

\begin{cl}\label{CL4}
Let $i \in W_n$. 
\begin{enumerate}
	\item For each $n$-choice function $c$, $\PA \vdash S_0(\num{i}) \to \PR_{g_0} \Bigl(\gn{ \bigvee_{i \prec_n V} S_0(\num{c(V)})} \Bigr)$. 
	Here, the empty disjunction denotes $0 = 1$. 
	
	\item For each $V \in \mathcal{P}(W_n)$ such that $i \prec_n V$, $\PA \vdash S_0(\num{i}) \to \neg \PR_{g_0} \Bigl(\gn{\neg \bigvee_{j \in V}S_0(\num{j})} \Bigr)$. 
\end{enumerate}
\end{cl}
\begin{proof}
1. Let $c$ be any $n$-choice function. 
We proceed in $\PA + S_0(\num{i})$: 
Let $m$ be a number such that $h_0(m)=0$ and $h_0(m+1)=i$. 
If $i \Vdash_n \Box \bot$, then the disjunction $\bigvee_{i \prec_n V} S_0(\num{c(V)})$ is empty and $g_0$ eventually outputs $0 = 1$, and hence $\PR_{g_0}(\gn{0=1})$ holds. 

If $i \nVdash_n \Box \bot$, then we shall show that $\neg \bigvee_{i \prec_n V} S_0(\num{c(V)}) \notin X_{m-1} \cup Y_{i, m-1}$. 

\begin{itemize}
\item Suppose $\neg \bigvee \nolimits_{i \prec_n V} S_0(\num{c(V)}) \in X_{m-1}$. 
Then, there exists some $\psi \in P_{T,m-1}$ such that $\psi \tto_{m-1} \neg \bigvee \nolimits_{i \prec_n V} S_0(\num{c(V)})$.
Since $i \nVdash_n \Box \bot$, we find a $U \in \mathcal{P}(W_n)$ such that $i \prec_n U$. 
Then, we obtain $\psi \tto_{m-1} \neg S_0(\num{c(U)})$ because $\neg \bigvee_{i \prec_n V} S_0(\num{c(V)}) \to \neg S_0(\num{c(U)})$ has a standard $T$-proof. 
We obtain that $\neg S_0(\num{c(U)})$ is a t.c.~of $P_{T, m-1}$, and this contradicts $h_0(m) = 0$. 
Therefore, $\neg \bigvee_{i \prec_n V} S_0(\num{c(V)}) \notin X_{m-1}$.

\item Suppose $\neg \bigvee_{i \prec_n V} S_0(\num{c(V)}) \in Y_{i, m-1}$. 
Then, there exists some $U \in \mathcal{P}(W_n)$ such that $i \prec_n U$ and $S_0(\num{j}) \tto_{m-1} \neg \bigvee_{i \prec_n V} S_0(\num{c(V)})$ for any $j \in U$. 
Since $c(U) \in U$, we have $S_0(\num{c(U)}) \tto_{m-1} \neg \bigvee_{i \prec_n V} S_0(\num{c(V)})$. 
Note that $U$ is a standard set because $(W_n, \prec_n)$ is a standard finite $\MN$-frame. 
Then, $\neg \bigvee_{i \prec_n V} S_0(\num{c(V)}) \to \neg S_0(\num{c(U)})$ has a standard $T$-proof, and hence we obtain $S_0(\num{c(U)}) \tto_{m-1} \neg S_0(\num{c(U)})$. 
It follows that $\neg S_0(\num{c(U)})$ is a t.c.~of $P_{T, m-1}$. 
This contradicts $h_0(m)=0$. 
We conclude $\neg \bigvee_{i \prec_n V} S_0(\num{c(V)}) \notin Y_{i, m-1}$. 
\end{itemize}
Therefore, we obtain $\neg \bigvee_{i \prec_n V} S_0(\num{c(V)}) \notin X_{m-1} \cup Y_{i, m-1}$. 
By Claim \ref{CL2}, we conclude that $\PR_{g_0} \Bigl(\gn{\bigvee_{i \prec_n V} S_0(\num{c(V)})} \Bigr)$ holds.

2. Let $V \in \mathcal{P}(W_n)$ be such that $i \prec_n V$. 
In this case, $i \nVdash_n \Box \bot$. 
We argue in $\PA$: Let $m$ be a number such that $h_0(m)=0$ and $h_0(m+1)=i$. 
For each $j' \in V$, we have $S_0(\num{j'}) \tto_{m-1} \neg \bigwedge_{j \in V} \neg S_0(\num{j})$, and thus we obtain $\neg \bigwedge_{j \in V} \neg S_0(\num{j}) \in Y_{i, m-1}$. 
By Claim \ref{CL2}, we conclude that $\neg \PR_{g_0} \Bigl(\gn{\bigwedge_{j \in V} \neg S_0(\num{j})} \Bigr)$.
\end{proof}

We define an arithmetical interpretation $f_{g_0}$ based on $\PR_{g_0}(x)$ by $f_{g_0}(p) \equiv \exists x \exists y (x \in W_y \land S_0(x) \land x \Vdash_y \gn{p})$ for each propositional variable $p$. 

\begin{cl}\label{CL5}
Let $i \in W_n$ and $A$ be any $\LB$-formula. 
\begin{enumerate}
	\item If $i \Vdash_n A$, then $\PA \vdash S_0(\num{i}) \to f_{g_0}(A)$. 

	\item If $i \nVdash_n A$, then $\PA \vdash S_0(\num{i}) \to \neg f_{g_0}(A)$. 
\end{enumerate}
\end{cl}
\begin{proof}
By induction on the construction of $A$, we prove these two statements simultaneously. 
We only prove the case $A \equiv \Box B$ for some $\LB$-formula $B$.

1. Suppose $i \Vdash_n \Box B$. 
Then, for any $V \in \mathcal{P}(W_n)$ satisfying $i \prec_n V$, there exists some $j \in V$ such that $j \Vdash_n B$.
By the induction hypothesis, we obtain $\PA \vdash S_0(\num{j})\to f_{g_0}(B)$. 
Thus, for some $n$-choice function $c$, $\PA \vdash \bigvee_{i \prec_n V}S_0(\num{c(V)}) \to f_{g_0}(B)$ holds. 
By Claim \ref{CL3}, $\PA \vdash \PR_{g_0} \Bigl(\gn{\bigvee_{i \prec_n V}S_0(\num{c(V)})} \Bigr) \to f_{g_0}(\Box B)$ holds.
Since $\PA \vdash S_0(\num{i}) \to \PR_{g_0} \Bigl(\gn{\bigvee_{i \prec_n V}S_0(\num{c(V)})} \Bigr)$ holds by Claim \ref{CL4}, we conclude that $\PA \vdash S_0(\num{i})\to f_{g_0}(\Box B)$.

2. Suppose $i \nVdash_n \square B$. 
Then, there exists some $V \in \mathcal{P}(W_n)$ such that $i \prec_n V$ and for all $j\in V$, $j \nVdash_n B$. 
By the induction hypothesis, we have $\PA \vdash \bigvee_{j\in V} S_0(\num{j}) \to \neg f_{g_0}(B)$, and hence, $\PA \vdash f_{g_0}(B) \to \bigwedge_{j\in V}\neg S_0(\num{j})$. 
By Claim \ref{CL3}, $\PA \vdash f_{g_0}(\Box B)\to \PR_{g_0} \Bigl(\gn{\bigwedge_{j\in V}\neg S_0(\num{j})} \Bigr)$ holds. 
By Claim \ref{CL4}, $\PA \vdash S_0(\num{i}) \to \neg \PR_{g_0} \Bigl(\gn{\bigwedge_{j\in V}\neg S_0(\num{j})} \Bigr)$ holds. 
Therefore, we conclude that $\PA \vdash S_0(\num{i}) \to \neg f_{g_0}(\Box B)$.
\end{proof}

We finish our proof of Theorem \ref{UACTMN}. 
The first clause of the theorem follows from Claims \ref{CL1}, \ref{CL3}, and \ref{CLF}. 

We prove the second clause.  
Suppose $L \nvdash A$. 
Then, there exist numbers $n \in \omega$ and $i \in W_n$ such that $i \nVdash_n A$. 
By Claim \ref{CL5}, $\PA \vdash S_0(\num{i}) \to \neg f_{g_0}(A)$ holds. 
If we suppose $T \vdash f_{g_0}(A)$, then $T \vdash \neg S_0(\num{i})$. 
This contradicts Proposition \ref{CLh}.3. 
Therefore, we conclude that $T\nvdash f_{g_0}(A)$. 
\end{proof}

\begin{cor}[The arithmetical completeness of $\MN$]\label{ACTMN}
For any $\LB$-formula $A$, the following are equivalent: 
	\begin{enumerate}
		\item $\MN \vdash A$. 
		\item $A \in \PL(\PR_T)$ for any provability predicate $\PR_T(x)$ of $T$ satisfying $\M$. 
		\item $A \in \PL(\PR_T)$ for any $\Sigma_1$ provability predicate $\PR_T(x)$ of $T$ satisfying $\M$. 
	\end{enumerate}
Moreover, there exists a $\Sigma_1$ provability predicate $\PR_T(x)$ of $T$ satisfying $\M$ such that $\PL(\PR_T) = \MN$. 
\end{cor}

\begin{cor}[The arithmetical completeness of $\MNF$]\label{ACTMNF}
For any $\LB$-formula $A$, the following are equivalent: 
	\begin{enumerate}
		\item $\MNF \vdash A$. 
		\item $A \in \PL(\PR_T)$ for any provability predicate $\PR_T(x)$ of $T$ satisfying $\M$ and $\D{3}$. 
		\item $A \in \PL(\PR_T)$ for any $\Sigma_1$ provability predicate $\PR_T(x)$ of $T$ satisfying $\M$ and $\D{3}$. 
	\end{enumerate}
Moreover, there exists a $\Sigma_1$ provability predicate $\PR_T(x)$ of $T$ satisfying $\M$ such that $\PL(\PR_T) = \MNF$. 
\end{cor}

\section{Arithmetical completeness of $\MNP$ and $\MNPF$}\label{Sec:MNP}

In this section, we prove the arithmetical completeness theorems for $\MNP$ and $\MNPF$ with respect to Rosser provability predicates. 
It is easily shown that $\MNP$ is arithmetically sound with respect to Rosser provability predicates satisfying the condition $\M$. 

\begin{prop}[The arithmetical soundness of $\MNP$ and $\MNPF$]
Let $\PR_T^{\mathrm{R}}(x)$ be a Rosser provability predicate of $T$ satisfying $\M$. 
Then, $\MNP \subseteq \PL(\PR_T^{\mathrm{R}})$ holds. 
Furthermore, $\PR_T^{\mathrm{R}}(x)$ satisfies $\D{3}$ if and only if $\MNPF \subseteq \PL(\PR_T^{\mathrm{R}})$. 
\end{prop}

As in the last section, we prove the following uniform version of the arithmetical completeness theorem. 

\begin{thm}[The uniform arithmetical completeness of $\MNP$ and $\MNPF$]\label{UACTMNP}
For $L \in \{\MNP, \MNPF\}$, there exists a Rosser provability predicate $\PR_T^{\mathrm{R}}(x)$ of $T$ satisfying $\M$ such that
\begin{enumerate}
	\item for any $\LB$-formula $A$ and any arithmetical interpretation $f$ based on $\PR_T^{\mathrm{R}}(x)$, if $L \vdash A$, then $\PA \vdash f(A)$, and
	\item there exists an arithmetical interpretation $f$ based on $\PR_T^{\mathrm{R}}(x)$ such that for any $\LB$-formula $A$, $L \vdash A$ if and only if $T \vdash f(A)$. 
\end{enumerate}
\end{thm}
\begin{proof}
Let $L \in \{\MNP, \MNPF\}$. 
By the proof of Theorem \ref{FFP}, we obtain a primitive recursively representable enumeration $\langle (W_n, \prec_n, \Vdash_n) \rangle_{n \in \omega}$ of pairwise disjoint finite $L$-models such that $W = \bigcup_{n \in \omega} W_n = \omega \setminus \{0\}$ and for any $\LB$-formula $A$, if $\MNP \nvdash A$, then there exist $n \in \omega$ and $i \in W_n$ such that $i \nVdash_n A$. 

As in the proof of Theorem \ref{UACTMN}, we simultaneously define primitive recursive functions $h_1$ and $g_1$ corresponding to Theorem \ref{UACTMNP} by using the double recursion theorem. 
Firstly, we define the function $h_1$. 
In the definition of $h_1$, the formula $\PR_{g_1}^{\mathrm{R}}(x) \equiv \exists y \bigl(\Fml(x) \land x = g_1(y) \wedge \forall z < y\, \dot{\neg}(x) \neq g_1(z) \bigr)$ is used. 
In fact, the definition of $h_1$ is exactly the same as that of $h_0$ defined in the proof of Theorem \ref{UACTMN} except that Rosser's predicate $\PR_{g_1}^{\mathrm{R}}(x)$ is used instead of the usual one. 

\begin{itemize}
	\item $h_1(0) = 0$. 
	\item $h_1(m+1) = \begin{cases}
				i & \text{if}\ h_1(m) = 0\\
					& \&\ i = \min \bigl\{j \in \omega \setminus \{0\} \mid \neg S_1(\num{j}) \ \text{is a t.c.~of}\ P_{T, m}\\
					& \quad\ \text{or}\ \exists \varphi \bigl[S_1(\num{j}) \to \neg \PR_{g_1}^{\mathrm{R}}(\gn{\varphi})\ \text{is a t.c.~of}\ P_{T, m}\\
					& \quad \quad \quad \& \ \neg \varphi \in F_m\ \&\ \neg \varphi \notin X_m \cup Y_{j, m} \bigr]\bigr\}\\
			h_1(m) & \text{otherwise}.
		\end{cases}$
\end{itemize}
Here, $S_1(x)$ is the $\Sigma_1$ formula $\exists y(h_1(y) = x)$. 
Also, for each $j \in W_n$ and number $m$, $X_m$ and $Y_{j, m}$ are sets defined as follows: 
\begin{itemize}
	\item $X_m : = \{\varphi \in F_m \mid \exists \psi \in P_{T, m}$ s.t.~$\psi \tto_m \varphi\}$. 
	\item $Y_{j, m} : = \{\varphi \in F_m \mid \exists V \in \mathcal{P}(W_n)$ s.t.~$j \prec_n V$ and $\forall l \in V\, (S_1(\num{l}) \tto_m \varphi)\}$. 
\end{itemize}

Secondly, we define the function $g_1$. 
The definition consists of Procedures 1 and 2, and Procedure 1 is exactly same as that of $g_0$. 
So, we only give the definition of Procedure 2.  

\vspace{0.1in}

\textsc{Procedure 2}

Suppose $m$ and $i \neq 0$ satisfy $h_1(m)=0$ and $h_1(m+1)=i$. 
Let $\chi_0, \chi_1, \ldots, \chi_{k-1}$ be an enumeration of all elements of $X_{m-1}$. 
For $l<k$, we define 
\begin{equation*}
	g_1(m+l)=\chi_l. 
\end{equation*} 
Let $\chi'_0, \chi'_1, \ldots, \chi'_{k'-1}$ be the enumeration of all elements of $Y_{i, m-1}$ in descending order of G\"{o}del numbers. 
For $l<k'$, we define 
\begin{equation*}
	g_1(m+k+l)=\chi'_{l}.  
\end{equation*}
For any $t$, we define
\begin{equation*}
g_1(m+k+k'+t)=\xi_{t}.
\end{equation*}

The formulas $\Prf_{g_1}(x,y)$, $\PR_{g_1}(x)$, and $\PR_{g_1}^{\mathrm{R}}(x)$ are defined as before. 
The following claim holds for $h_1$ as well as for the function $h_0$.

\begin{cl}\label{PCLh}
\leavevmode
\begin{enumerate}
	\item $\PA \vdash \forall x \forall y(0 < x < y \land S_1(x) \to \neg S_1(y))$. 
	\item $\PA \vdash \Prov_T(\gn{0=1}) \leftrightarrow \exists x(S_1(x) \land x \neq 0)$. 
	\item For each $i \in \omega \setminus \{0\}$, $T \nvdash \neg S_1(\num{i})$. 
	\item For each $m \in \omega$, $\PA \vdash \forall x \forall y(h_1(x) = 0 \land h_1(x+1) = y \land y \neq 0 \to x > \num{m})$.
\end{enumerate}
\end{cl}
\begin{proof}
The use of the Rosser provability predicate $\PR_{g_1}^{\mathrm{R}}(x)$ is the point where the definition of the function $h_1$ differs from that of $h_0$. 
Since that difference affects the proof of clause 2 of this claim, we prove only clause 2. 
The implication $(\to)$ is easy, and so we prove $(\leftarrow)$. 
We proceed in $\PA$: 
Suppose that $S_1(i)$ holds for some $i \neq 0$. 
Let $m$ and $n$ be such that $h_1(m) = 0$, $h_1(m+1) = i$, and $i \in W_n$. 
We would like to show the inconsistency of $T$. 
By the definition of $h_1$, we distinguish the following two cases: 

Case 1: $\neg S_1(\num{i})$ is a t.c.~of $P_{T, m}$. \\
	Then, $\neg S_1(\num{i})$ is $T$-provable. 
	On the other hand, $S_1(\num{i})$ is $T$-provable because it is a true $\Sigma_1$ sentence. 
	Therefore, $T$ is inconsistent. 

Case 2: There exists a $\varphi$ such that $S_1(\num{i}) \to \neg \PR_{g_1}^{\mathrm{R}}(\gn{\varphi})$ is a t.c.~of $P_{T, m}$, $\neg \varphi \in F_m$, and $\neg \varphi \notin X_m \cup Y_{i, m}$. \\
	Let $s$ and $u$ be numbers such that $\xi_s \equiv \varphi$ and $\xi_u \equiv \neg \varphi$. 
	Then, $g_1(m + k + k' + s) = \varphi$ and $g_1(m + k + k' + u) = \neg \varphi$, where $k$ and $k'$ are the cardinalities of $X_{m-1}$ and $Y_{i, m-1}$, respectively.  
	Since the G\"odel number of $\varphi$ is smaller than that of $\neg \varphi$, we have $s < u$. 
	Since the relation $\tto_{m}$ is reflexive, we have that $P_{T, m} \subseteq X_{m}$. 
	It follows $\neg \varphi \notin P_{T, m}$ because $\neg \varphi \notin X_{m}$, and hence $\neg \varphi$ is not output in Procedure 1. 
	Since $\neg \varphi \notin X_{m-1} \cup Y_{i, m-1}$, by the definition of Procedure 2, we have that $g_1(m + k + k' + u)$ is the first output of $\neg \varphi$ by $g_1$. 
	Thus, $\varphi$ is output before any output of $\neg \varphi$, that is, $\PR_{g_1}^{\mathrm{R}}(\gn{\varphi})$ holds. 
	Then, $S_1(\num{i}) \land \PR_{g_1}^{\mathrm{R}}(\gn{\varphi})$ is a true $\Sigma_1$ sentence, and so it is provable in $T$.
	Since $S_1(\num{i}) \to \neg \PR_{g_1}^{\mathrm{R}}(\gn{\varphi})$ is also $T$-provable, $T$ is inconsistent.
\end{proof}

\begin{cl}\label{PCL1}\leavevmode
\begin{enumerate}
\item $\PA \vdash \forall x \bigl(\Prov_{T}(x) \leftrightarrow \PR_{g_1}(x) \bigr)$.
\item For any $n \in \omega$ and any $\LA$-formula $\varphi$, $\PA \vdash \Proof_{T}(\gn{\varphi},\num{n}) \leftrightarrow \Prf_{g_1}(\gn{\varphi},\num{n})$.
\end{enumerate}
\end{cl}
\begin{proof}
1. By the definition of $g_1$, 
\[
	\PA \vdash \neg \exists x(S_1(x) \wedge x \neq 0) \rightarrow \forall x \bigl(\Prov_{T}(x)\leftrightarrow \PR_{g_1}(x) \bigr)
\]
holds.   
Since $g_1$ outputs all formulas in Procedure 2, 
\[
	\PA \vdash \exists x(S_1(x) \wedge x \neq 0) \rightarrow \forall x \bigl(\Fml(x) \leftrightarrow \PR_{g_1}(x) \bigr). 
\]
Also, $\PA \vdash \Prov_T(\gn{0=1}) \rightarrow \forall x \bigl(\Prov_{T}(x) \leftrightarrow \Fml(x) \bigr)$. 
Since $\Prov_T(\gn{0=1})$ and $\exists x (S_1(x) \land x \neq 0)$ are equivalent in $\PA$ by Proposition \ref{PCLh}.2, we obtain 
\[
	\PA \vdash \exists x(S_1(x) \wedge x \neq 0) \rightarrow \forall x \bigl(\Prov_{T}(x) \leftrightarrow \PR_{g_1}(x) \bigr).
\]
By the law of excluded middle, we conclude $\PA \vdash \forall x \bigl(\Prov_{T}(x) \leftrightarrow \PR_{g_1}(x) \bigr)$.

2. This is proved as in the proof of Claim \ref{CL1}.2. 
\end{proof}

Then, our $\Sigma_1$ formula $\PR_{g_1}^{\mathrm{R}}(x)$ is a Rosser provability predicate of $T$. 

\begin{cl}\label{PCL2}
$\PA$ proves the following statement:
``Suppose $m$ and $i \neq 0$ satisfy $h_1(m)=0$ and $h_1(m+1)=i$. 
Then, for any $\LA$-formula $\varphi$, $\neg \varphi \in X_{m-1} \cup Y_{i, m-1}$ if and only if $\neg \PR_{g_1}^{\mathrm{R}}(\gn{\varphi})$ holds".
\end{cl}
\begin{proof}
We argue in $\PA$: Let $m$ and $i \neq 0$ satisfy $h_1(m)=0$ and $h_1(m+1)=i$.  

$(\Rightarrow)$: Suppose $\neg \varphi \in X_{m-1} \cup Y_{i, m-1}$. 
We show that $g_1$ outputs $\neg \varphi$ before it outputs $\varphi$. 
We distinguish the following two cases: 

Case 1: $\neg \varphi \in X_{m-1}$. \\
In this case, we have $\neg \varphi \in \{g_1(m), \ldots, g_1(m + k-1)\}$, where $k$ is the cardinality of $X_{m-1}$. 
We would like to show $\varphi \notin \{g_1(0), g_1(1), \ldots, g_1(m + k-1)\}$, that is, $\varphi \notin P_{T,m-1} \cup X_{m-1}$. 
Since $P_{T, m-1} \subseteq X_{m-1}$, it suffices to show $\varphi \notin X_{m-1}$. 
Suppose, towards a contradiction, that $\varphi \in X_{m-1}$. 
Then, there exists some $\psi \in P_{T,m-1}$ such that $\psi \tto_{m-1} \varphi$. 
Since $\neg \varphi \in X_{m-1}$, there exists some $\rho \in P_{T, m-1}$ such that $\rho \tto_{m-1} \neg \varphi$.
Then, $P_{T, m-1}$ is not propositionally satisfiable, and hence $\neg S_1(\num{1})$ is a t.c.~of $P_{T, m-1}$. 
This contradicts $h_1(m)=0$. 
Hence $\varphi \notin X_{m-1}$. 

Case 2: $\neg \varphi \in Y_{i, m-1}$. \\
Firstly, we show $\varphi \notin X_{m-1}$. 
Suppose, towards a contradiction, that $\varphi \in X_{m-1}$. 
Then, $\psi \tto_{m-1} \varphi$ holds for some $\psi \in P_{T,m-1}$. 
Let $n$ be such that $i \in W_n$. 
Since $\neg \varphi \in Y_{i, m-1}$, there exists some $V \in \mathcal{P}(W_n)$ such that $i \prec_n V$ and $S_1(\num{j}) \tto_{m-1} \neg \varphi$ for any $j \in V$. 
Since $V$ is non-empty, we find some $j_0 \in V$, and hence $S_1(\num{j_0}) \tto_{m-1} \neg \varphi$ holds. 
Thus, $\neg S_1(\num{j_0})$ is a t.c.~of $P_{T, m-1}$, and this contradicts $h_1(m)=0$. 
Therefore, we obtain $\varphi \notin X_{m-1}$, and hence $\varphi  \notin P_{T,m-1}\cup X_{m-1}$. 
This means that $\varphi \notin \{g_1(0), \ldots, g_1(m+k-1)\}$. 

Since $\neg \varphi \in Y_{i, m-1}$, $g_1(m + k + l) = \neg \varphi$ for some $l < k'$, where $k'$ is the cardinality of $Y_{i, m-1}$. 
Thus, it suffices to show that $\varphi \notin \{g_1(m + k), \ldots, g_1(m+k+l)\}$. 
If $\varphi \notin Y_{i, m-1}$, we are done. 
If $\varphi \in Y_{i, m-1}$, then the first output of $\varphi$ is $g_1(m + k + l') = \varphi$ for some $l' > l$ because the G\"{o}del number of $\varphi$ is smaller than that of $\neg \varphi$. 
Thus, $g_1$ outputs $\neg \varphi$ before outputting $\varphi$. 

$(\Leftarrow)$: Suppose that $g_1$ outputs $\neg \varphi$ before outputting $\varphi$. 
Since the G\"{o}del number of $\varphi$ is smaller than that of $\neg \varphi$, the first time $\neg \varphi$ is output is not when the elements of the enumeration $\langle \xi_t \rangle_{t\in\omega}$ are output in Procedure 2. 
Thus, $\neg \varphi \in P_{T,m-1} \cup X_{m-1} \cup Y_{i, m-1} = X_{m-1} \cup Y_{i, m-1}$. 
\end{proof}

We show that our Rosser provability predicate $\PR_{g_1}^{\mathrm{R}}(x)$ satisfies the condition $\M$. 

\begin{cl}\label{PCL3}
Let $\varphi$ and $\psi$ be any $\LA$-formulas. 
If $T \vdash \varphi \to \psi$, then $\PA \vdash \PR_{g_1}^{\mathrm{R}}(\gn{\varphi}) \to \PR_{g_1}^{\mathrm{R}}( \gn{\psi})$.
\end{cl}
\begin{proof}
Suppose $T \vdash \varphi \to \psi$. 
Then, $\PA \vdash \Prov_T(\gn{\varphi}) \to \Prov_T(\gn{\psi})$.
Since $\PA + \neg \exists x (S_1(x)\wedge x\neq 0) \vdash \forall x \bigl(\Prov_T(x) \leftrightarrow \PR_{g_1}^{\mathrm{R}}(x) \bigr)$, 
we have $\PA + \neg \exists x (S_1(x)\wedge x\neq 0) \vdash \PR_{g_1}^{\mathrm{R}}(\gn{\varphi}) \to \PR_{g_1}^{\mathrm{R}}(\gn{\psi})$. 

On the other hand, $\PA + \exists x (S_1(x)\wedge x\neq 0) \vdash \PR_{g_1}^{\mathrm{R}}(\gn{\varphi}) \to \PR_{g_1}^{\mathrm{R}}(\gn{\psi})$ is proved in the same way as in the proof of Claim \ref{CL3} by using Claim \ref{PCL2}. 
\end{proof}

When $L = \MNPF$, it is shown that $\PR_{g_1}^{\mathrm{R}}(x)$ satisfies the condition $\D{3}$. 

\begin{cl}\label{CLPF}
If $L = \MNPF$, then for any $\LA$-formula $\varphi$, $\PA \vdash \PR_{g_1}^{\mathrm{R}}(\gn{\varphi}) \to \PR_{g_1}^{\mathrm{R}}(\gn{\PR_{g_1}^{\mathrm{R}}(\gn{\varphi})})$. 
\end{cl}
\begin{proof}
Since $\PR_{g_1}^{\mathrm{R}}(\gn{\varphi})$ is a $\Sigma_1$ sentence, $\PA \vdash \PR_{g_1}^{\mathrm{R}}(\gn{\varphi}) \to \Prov_T(\gn{\PR_{g_1}^{\mathrm{R}}(\gn{\varphi})})$. 
By Claim \ref{PCL1}.1, $\PA \vdash \PR_{g_1}^{\mathrm{R}}(\gn{\varphi}) \to \PR_{g_1}(\gn{\PR_{g_1}^{\mathrm{R}}(\gn{\varphi})})$. 

We show $\PA + \neg \exists x (S_1(x)\wedge x\neq 0) \vdash \forall x \bigl(\PR_{g_1}(x) \leftrightarrow \PR_{g_1}^{\mathrm{R}}(x) \bigr)$ by reasoning in $\PA + \neg \exists x (S_1(x)\wedge x\neq 0)$: 
Since the construction of $g_1$ never swtches to Procedure 2, we have that $g_1(m) = \varphi$ if and only if $m$ is a $T$-proof of $\varphi$. 
By Claim \ref{PCLh}.3, we have that $T$ is consistent. 
Thus, $g_1(m) = \varphi$ if and only if $g_1(m) = \varphi$ and $g_1(k) \neq \neg \varphi$ for all $k < m$. 
This means that $\forall x \bigl(\PR_{g_1}(x) \leftrightarrow \PR_{g_1}^{\mathrm{R}}(x) \bigr)$ holds. 

Then, we have $\PA + \neg \exists x (S_1(x)\wedge x\neq 0) \vdash \PR_{g_1}^{\mathrm{R}}(\gn{\varphi}) \to \PR_{g_1}^{\mathrm{R}}(\gn{\PR_{g_1}^{\mathrm{R}}(\gn{\varphi})})$. 
On the other hand, $\PA + \exists x (S_1(x)\wedge x\neq 0) \vdash \PR_{g_1}^{\mathrm{R}}(\gn{\varphi}) \to \PR_{g_1}^{\mathrm{R}}(\gn{\PR_{g_1}^{\mathrm{R}}(\gn{\varphi})})$ is proved in the similar way as in the proof of Claim \ref{CLF}. 
Finally, by the law of excluded middle, we conclude $\PA \vdash \PR_{g_1}^{\mathrm{R}}(\gn{\varphi}) \to \PR_{g_1}^{\mathrm{R}}(\gn{\PR_{g_1}^{\mathrm{R}}(\gn{\varphi})})$. 
\end{proof}

The following claim is proved by using Claim \ref{PCL2} in the similar way as in the proof of Claim \ref{CL4}. 

\begin{cl}\label{PCL4}
Let $i \in W_n$. 
\begin{enumerate}
\item For any $n$-choice function $c$, $\PA \vdash S_1(\num{i}) \rightarrow \PR_{g_1}^{\mathrm{R}} \Bigl(\gn{\bigvee_{i\prec_n V} S_1(\num{c(V)})} \Bigr)$.
\item For any $V \in \mathcal{P}(W_n)$ satisfying $i \prec_n V$, $\PA \vdash S_1(\num{i}) \rightarrow \neg \PR_{g_1}^{\mathrm{R}} \Bigl(\gn{\bigwedge_{j \in V} \neg S_1(\num{j})} \Bigr)$.
\end{enumerate}
\end{cl}

Let $f_{g_1}$ be the arithmetical interpretation based on $\PR_{g_1}^{\mathrm{R}}(x)$ defined by $f_{g_1}(p) \equiv \exists x \exists y (x \in W_y \land S_1(x) \land x \Vdash_y \gn{p})$. 
Then, the following claim is also proved as in the proof of Claim \ref{CL5} by using Claim \ref{PCL4}. 

\begin{cl}\label{PCL5}
For any $i \in W_n$ and any $\LB$-formula $A$, the following hold:
\begin{enumerate}
\item If $i \Vdash_n A$, then $\PA \vdash S_1(\num{i}) \to f_{g_1}(A)$. 
\item If $i \nVdash_n A$, then $\PA \vdash S_1(\num{i})\to \neg f_{g_1}(A)$.
\end{enumerate}
\end{cl}

We finish our proof of the theorem. 
The first clause follows from Claims \ref{PCL1}, \ref{PCL3}, and \ref{CLPF}.  
The second clause is proved by using Claims \ref{PCLh}.3 and \ref{PCL5}. 
\end{proof}

\begin{cor}[The arithmetical completeness of $\MNP$]\label{ACTMNP}
For any $\LB$-formula $A$, the following are equivalent: 
	\begin{enumerate}
		\item $\MNP \vdash A$. 
		\item $A \in \PL(\PR_T^{\mathrm{R}})$ for any Rosser provability predicate $\PR_T^{\mathrm{R}}(x)$ of $T$ satisfying $\M$. 
	\end{enumerate}
Moreover, there exists a Rosser provability predicate $\PR_T^{\mathrm{R}}(x)$ of $T$ satisfying $\M$ such that $\PL(\PR_T^{\mathrm{R}}) = \MNP$. 
\end{cor}

\begin{cor}[The arithmetical completeness of $\MNPF$]\label{ACTMNPF}
For any $\LB$-formula $A$, the following are equivalent: 
	\begin{enumerate}
		\item $\MNPF \vdash A$. 
		\item $A \in \PL(\PR_T^{\mathrm{R}})$ for any Rosser provability predicate $\PR_T^{\mathrm{R}}(x)$ of $T$ satisfying $\M$ and $\D{3}$. 
	\end{enumerate}
Moreover, there exists a Rosser provability predicate $\PR_T^{\mathrm{R}}(x)$ of $T$ satisfying $\M$ such that $\PL(\PR_T^{\mathrm{R}}) = \MNPF$. 
\end{cor}

In \cite{Kur21}, the existence of a Rosser provability predicate of $T$ satisfying $\M$ and $\D{3}$ is proved. 
Corollary \ref{ACTMNPF} is a strengthening of the result.

\section{Arithmetical completeness of $\MND$}\label{Sec:MND}

In this section, we investigate Rosser provability predicates $\PR_T^{\mathrm{R}}(x)$ such that the schematic consistency statement $\{\neg \bigl(\PR_T^{\mathrm{R}}(\gn{\varphi}) \land \PR_T^{\mathrm{R}}(\gn{\neg \varphi}) \bigr) \mid \varphi$ is an $\LA$-formula$\}$ is provable. 

\begin{prop}[The arithmetical soundness of $\MND$]
Let $\PR_T^{\mathrm{R}}(x)$ be a Rosser provability predicate of $T$ satisfying $\M$. 
Then, $\{\neg \bigl(\PR_T^{\mathrm{R}}(\gn{\varphi}) \land \PR_T^{\mathrm{R}}(\gn{\neg \varphi}) \bigr)\}$ is provable in $T$ if and only if $\MND \subseteq \PL(\PR_T^{\mathrm{R}})$. 
\end{prop}

We prove the existence of a Rosser provability predicate $\PR_T^{\mathrm{R}}(x)$ of $T$ exactly corresponding to $\MND$. 
The idea of our proof of the theorem follows almost the same as that of the proof of Theorem \ref{UACTMNP}.
However, Claim \ref{PCL2} in the proof of Theorem \ref{UACTMNP} does not hold for a proof we require in this section. 
This is because if the equivalence $\neg \varphi \in X_{m-1} \cup Y_{i, m-1} \iff \neg \PR_T^{\mathrm{R}}(\gn{\varphi})$ holds for any $\varphi$, then for some formula $\varphi$ such that $\neg \varphi, \neg \neg \varphi \notin X_{m-1} \cup Y_{i, m-1}$, $\neg \bigl(\PR_T^{\mathrm{R}}(\gn{\varphi}) \land \PR_T^{\mathrm{R}}(\gn{\neg \varphi}) \bigr)$ does not hold. 
This is contrary to the requirement that $\PR_T^{\mathrm{R}}(x)$ corresponds to $\MND$. 
Therefore, we change our strategy of constructing a function enumerating all theorems of $T$ so that the equivalence $\varphi \in X_{m-1} \cup Y_{i, m-1} \iff \PR_T^{\mathrm{R}}(\gn{\varphi})$ holds. 

In this section, we use a primitive recursive function $h$ introduced in the paper \cite{Kur20_1} instead of $h_0$ and $h_1$. 
In that paper, the arithmetical completeness theorem for the normal modal logic $\mathsf{KD}$ is proved by using the function $h$. 
Also, in the paper \cite{Kur}, the arithmetical completeness theorems of extensions of the non-normal modal logic $\mathsf{N}$ are also proved by using $h$. 
As a matter of fact, the arithmetical completeness of $\MN$ and $\MNP$ can be proved by using $h$ instead of $h_0$ and $h_1$, respectively. 
The advantage of using the function $h$ here is that the definition of $h$ is simpler and Proposition \ref{Prop:h} below has already been established.

The function $h$ is defined as follows by using the recursion theorem: 
\begin{itemize}
	\item $h(0) = 0$. 
	\item $h(m+1) = \begin{cases} i & \text{if}\ h(m) = 0\\
				& \quad \&\ i = \min \{j \in \omega \setminus \{0\} \mid \neg S(\num{j}) \ \text{is a t.c.~of}\ P_{T, m}\}, \\
			h(m) & \text{otherwise}.
		\end{cases}$
\end{itemize}
Here, $S(x)$ is the $\Sigma_1$ formula $\exists y(h(y) = x)$. 
Then, the following proposition holds: 

\begin{prop}[Cf.~{\cite[Lemma 3.2.]{Kur20_1}}]\label{Prop:h}
\leavevmode
\begin{enumerate}
	\item $\PA \vdash \forall x \forall y(0 < x < y \land S(x) \to \neg S(y))$. 
	\item $\PA \vdash \Prov_T(\gn{0=1}) \leftrightarrow \exists x(S(x) \land x \neq 0)$. 
	\item For each $i \in \omega \setminus \{0\}$, $T \nvdash \neg S(\num{i})$. 
	\item For each $m \in \omega$, $\PA \vdash \forall x \forall y(h(x) = 0 \land h(x+1) = y \land y \neq 0 \to x \geq \num{m})$. 
\end{enumerate}
\end{prop}

\begin{thm}[The uniform arithmetical completeness of $\MND$]\label{UACTMND}
There exists a Rosser provability predicate $\PR_T^{\mathrm{R}}(x)$ of $T$ satisfying $\M$ such that
\begin{enumerate}
	\item for any $\LB$-formula $A$ and any arithmetical interpretation $f$ based on $\PR_T^{\mathrm{R}}(x)$, if $\MND \vdash A$, then $\PA \vdash f(A)$, and
	\item there exists an arithmetical interpretation $f$ based on $\PR_T^{\mathrm{R}}(x)$ such that for any $\LB$-formula $A$, $\MND \vdash A$ if and only if $T \vdash f(A)$. 
\end{enumerate}
\end{thm}
\begin{proof}
As in the proof of Theorem \ref{UACTMN}, we obtain a primitive recursively representable enumeration $\langle (W_n, \prec_n, \Vdash_n) \rangle_{n \in \omega}$ of pairwise disjoint finite $\MND$-models such that $W = \bigcup_{n \in \omega} W_n = \omega \setminus \{0\}$ and for any $\LB$-formula $A$, if $\MND \nvdash A$, then there exist $n \in \omega$ and $i \in W_n$ such that $i \nVdash_n A$. 

For each $j \in W_n$ and number $m$, we define the finite sets $X_m$ and $Y_{j, m}$ as follows: 
\begin{itemize}
	\item $X_m : = \{ \varphi \in F_m \mid \exists \psi \in P_{T, m}$ s.t.~$\psi \tto_{m} \varphi\}$.  
	\item $Y_{j, m} : = \{ \varphi \in F_m \mid \exists c$: $n$-choice function s.t.~$\forall U \in \mathcal{P}(W_n)\, (j \prec_n U \Rightarrow S(\num{c(U)}) \tto_{m} \varphi)\}$. 
\end{itemize}

We define a primitive recursive function $g_2$ corresponding to Theorem \ref{UACTMND}. 
As in the proof of Theorem \ref{UACTMNP}, we only give the definition of Procedure 2. 

\vspace{0.1in}

\textsc{Procedure 2}.

Suppose $m$ and $i \neq 0$ satisfy $h(m)=0$ and $h(m+1)=i$. 
Let $n$ be a number such that $i \in W_n$. 
Let $\chi_0, \chi_1, \ldots, \chi_{k-1}$ be an enumeration of all elements of $X_{m-1} \cup Y_{i, m-1}$. 
For $l<k$, we define 
\begin{equation*}
	g_2(m+l)=\chi_l. 
\end{equation*} 
For $t$ and $s < m$, we define
\begin{equation*}
g_2(m+k+mt + s) = \overbrace{\neg \cdots \neg}^{m-s-1} \xi_{t}.
\end{equation*}

The definition of $g_2$ has just been finished. 
We define the formulas $\Prf_{g_2}(x,y)$, $\PR_{g_2}(x)$, and $\PR_{g_2}^{\mathrm{R}}(x)$ as before. 
The proof of the following claim is completely same as that of Claim \ref{PCL1}. 

\begin{cl}\label{DCL1}\leavevmode
\begin{enumerate}
\item $\PA \vdash \forall x \bigl(\Prov_{T}(x) \leftrightarrow \PR_{g_2}(x) \bigr)$.
\item For any $n \in \omega$ and any $\LA$-formula $\varphi$, $\PA \vdash \Proof_{T}(\gn{\varphi},\num{n}) \leftrightarrow \Prf_{g_2}(\gn{\varphi},\num{n})$.
\end{enumerate}
\end{cl}

\begin{cl}\label{DCL2}
$\PA$ proves the following statement: 
``Let $m$ and $i \neq 0$ be such that $h(m) = 0$ and $h(m + 1) = i$. 
Then, for any $\LA$-formula $\varphi$, $\varphi \notin X_{m-1} \cup Y_{i, m-1}$ or $\neg \varphi \notin X_{m-1} \cup Y_{i, m-1}$''. 
\end{cl}
\begin{proof}
We reason in $\PA$: 
Let $m$ and $i$ be as in the statement. 
Also, let $n$ be such that $i \in W_n$. 
Suppose, towards a contradiction, that $\varphi \in X_{m-1} \cup Y_{i, m-1}$ and $\neg \varphi \in X_{m-1} \cup Y_{i, m-1}$. 
We distinguish the following four cases: 

Case 1: $\varphi \in X_{m-1}$ and $\neg \varphi \in X_{m-1}$. \\
There exist $\psi, \rho \in P_{T, m-1}$ such that $\psi \tto_{m-1} \varphi$ and $\rho \tto_{m-1} \neg \varphi$. 
Then, $P_{T, m-1}$ is not propositionally satisfiable.  
This contradicts $h(m) = 0$.  

Case 2: $\varphi \in Y_{i, m-1}$ and $\neg \varphi \in X_{m-1}$. \\
There exist an $n$-choice function $c$ and $\rho \in P_{T, m-1}$ such that $S(\num{c(U)}) \tto_{m-1} \varphi$ for any $U \in \mathcal{P}(W_n)$ with $i \prec_n U$ and $\rho \tto_{m-1} \neg \varphi$. 
Since $\neg (i \prec_n \emptyset)$ and $(W_n, \prec_n)$ is an $\MND$-frame, we have $i \prec_n W_n$. 
Then, $S(\num{c(W_n)}) \tto_{m-1} \varphi$. 
We obtain that $\neg S(\num{c(W_n)})$ is a t.c.~of $P_{T, m-1}$, and this is a contradiction. 

Case 3: $\varphi \in X_{m-1}$ and $\neg \varphi \in Y_{i, m-1}$. \\
Similarly as in Case 2, we have $h(m) \neq 0$, and this is a contradiction. 

Case 4: $\varphi \in Y_{i, m-1}$ and $\neg \varphi \in Y_{i, m-1}$. \\
There exist $n$-choice functions $c_0$ and $c_1$ satisfying the following conditions: 
\begin{itemize}
	\item $S(\num{c_0(U)}) \tto_{m-1} \varphi$ for any $U \in \mathcal{P}(W_n)$ with $i \prec_n U$; 
	\item $S(\num{c_1(U)}) \tto_{m-1} \neg \varphi$ for any $U \in \mathcal{P}(W_n)$ with $i \prec_n U$. 
\end{itemize}	 
Since $i \prec_n W_n$, we have $S(\num{c_0(W_n)}) \tto_{m-1} \varphi$ and $S(\num{c_1(W_n)}) \tto_{m-1} \neg \varphi$.  
We define an increasing sequence $Z_0 \subseteq Z_1 \subseteq \cdots$ of subsets of $W_n$ inductively as follows: 
\begin{itemize}
	\item $Z_0 : = \{c_0(W_n)\}$; 
	\item $Z_{q+1} : = Z_q \cup \{c_0(W_n \setminus Z_q)\}$. 
\end{itemize}
For each $q$, if $W_n \setminus Z_q \neq \emptyset$, then $c_0(W_n \setminus Z_q) \in W_n \setminus Z_q$, so we have $Z_q \subsetneq Z_{q+1}$. 
Hence, $Z_{p-1} = W_n$ for the cardinality $p$ of $W_n$. 

Here, we prove that for any $q < p-1$ and any $j \in Z_q$, $S(\num{j}) \tto_{m-1} \varphi$ by induction on $q$. 
\begin{itemize}
	\item For $q = 0$, the statement holds because $Z_0 = \{c_0(W_n)\}$ and $S(\num{c_0(W_n)}) \tto_{m-1} \varphi$. 

	\item Assume that the statement holds for $q$ and that $q + 1 < p$. 
	If $i \prec_n Z_q$, then $S(\num{c_1(Z_q)}) \tto_{m-1} \neg \varphi$. 
	On the other hand, by the induction hypothesis, we have $S(\num{c_1(Z_q)}) \tto_{m-1} \varphi$ because $c_1(Z_q) \in Z_q$. 
	Thus, $\neg S(\num{c_1(Z_q)})$ is a t.c.~of $P_{T, m-1}$, and this contradicts $h(m) = 0$. 
	Hence, we obtain $\neg (i \prec_n Z_q)$. 
	Since $(W_n, \prec_n)$ is an $\MND$-frame, we get $i \prec_n (W_n \setminus Z_q)$. 
	Then, we have $S(\num{c_0(W_n \setminus Z_q)}) \tto_{m-1} \varphi$. 
	Therefore, the set $Z_{q+1} = Z_q \cup \{c_0(W_n \setminus Z_q)\}$ satisfies the required condition. 
\end{itemize}
	In particular, for $Z_{p-1} = W_n$, we have that for any $j \in W_n$, $S(\num{j}) \tto_{m-1} \varphi$. 
	On the other hand, since $c_1(W_n) \in W_n$ and $S(\num{c_1(W_n)}) \tto_{m-1} \neg \varphi$, we have that $\neg S(\num{c_1(W_n)})$ is a t.c.~of $P_{T, m-1}$. 
	This is a contradiction. 

Therefore, we conclude that $\varphi \notin X_{m-1} \cup Y_{i, m-1}$ or $\neg \varphi \notin X_{m-1} \cup Y_{i, m-1}$. 
\end{proof}

\begin{cl}\label{DCL3}
For any $\LA$-formula $\varphi$, $\PA$ proves the following statement: 
``Let $m$ and $i \neq 0$ be such that $h(m) = 0$ and $h(m + 1) = i$. 
Then, $\varphi \in X_{m-1} \cup Y_{i, m-1}$ if and only if $\PR_{g_2}^{\mathrm{R}}(\gn{\varphi})$ holds''.
\end{cl}
\begin{proof}
Let $\varphi$ be any $\LA$-formula. 
We argue in $\PA$: 
Let $m$ and $i$ be as in the statement of the claim. 

$(\Rightarrow)$: 
Suppose $\varphi \in X_{m-1} \cup Y_{i, m-1}$. 
Then, by the definition of $g_2$, we have $\varphi \in \{g_2(m), \ldots, g_2(m + k-1)\}$, where $k$ is the cardinality of $X_{m-1} \cup Y_{i, m-1}$. 
On the other hand, by Claim \ref{DCL2}, $\neg \varphi \notin X_{m-1} \cup Y_{i, m-1}$. 
Also, $\neg \varphi \notin P_{T, m-1}$. 
By the definition of $g_2$, we obtain $\neg \varphi \notin \{g_2(0), \ldots, g_2(m+k-1)\}$. 
Therefore, $\PR_{g_2}^{\mathrm{R}}(\gn{\varphi})$ holds. 

$(\Leftarrow)$: 
Suppose $\varphi \notin X_{m-1} \cup Y_{i, m-1}$. 
Then, $\varphi \notin P_{T, m-1}$. 
By the definition of $g_2$, $\varphi \notin \{g_2(0), \ldots, g_2(m+k-1)\}$. 
Let $\psi$ be the formula obtained by deleting all leading $\neg$'s from $\varphi$, and let $u$ be the number of deleted $\neg$'s from $\varphi$. 
Then, by Proposition \ref{Prop:h}.4, we have $m \geq u + 2$ (because $\varphi$ is a standard formula).   
Thus, for the unique $t$ with $g_2(m + k + mt + m - 1) = \psi$, we obtain $g_2(m + k + mt + m - u - 2) = \neg \varphi$ and $g_2(m + k + mt + m - u - 1) = \varphi$. 
Moreover, $g_2(m + k + mt + m - u - 1) = \varphi$ is the first output of $\varphi$, and hence $\neg \varphi$ is output before outputting $\varphi$. 
We conclude that $\neg \PR_{g_2}^{\mathrm{R}}(\gn{\varphi})$ holds. 
\end{proof}

\begin{cl}\label{DCL4}
For any $\LA$-formulas $\varphi$ and $\psi$, if $T \vdash \varphi \to \psi$, then $\PA \vdash \PR_{g_2}^{\mathrm{R}}(\gn{\varphi}) \to \PR_{g_2}^{\mathrm{R}}(\gn{\psi})$. 
\end{cl}
\begin{proof}
Suppose $T \vdash \varphi \to \psi$. 
Since $\PA \vdash \Prov_T(\gn{\varphi}) \to \Prov_T(\gn{\psi})$, we obtain $\PA + \neg \exists x (S(x) \land x \neq 0) \vdash \PR_{g_2}^{\mathrm{R}}(\gn{\varphi}) \to \PR_{g_2}^{\mathrm{R}}(\gn{\psi})$ as in the proof of Claim \ref{PCL4}. 
It suffices to prove $\PA + \exists x (S(x) \land x \neq 0) \vdash \PR_{g_2}^{\mathrm{R}}(\gn{\varphi}) \to \PR_{g_2}^{\mathrm{R}}(\gn{\psi})$. 

We reason in $\PA + \exists x (S(x) \land x \neq 0)$: 
Let $m$ and $i \neq 0$ be such that $h(m) = 0$ and $h(m+1) = i$. 
By the supposition, we have $\varphi \to \psi \in P_{T, m-1}$. 
Assume that $\PR_{g_2}^{\mathrm{R}}(\gn{\varphi})$ holds. 
Then, by Claim \ref{DCL3}, we have $\varphi \in X_{m-1} \cup Y_{i, m-1}$. 
We distinguish the following two cases: 

Case 1: $\varphi \in X_{m-1}$. \\
There exists some $\rho \in P_{T, m-1}$ such that $\rho \tto_{m-1} \varphi$. 
Then, $\rho \tto_{m-1} \psi$, and hence $\psi \in X_{m-1}$. 

Case 2: $\varphi \in Y_{i, m-1}$. \\
Let $n$ be such that $i \in W_n$. 
Then, there exists an $n$-choice function $c$ satisfying $S(\num{c(U)}) \tto_{m-1} \varphi$ for any $U \in \mathcal{P}(W_n)$ with $i \prec_n U$. 
Then $S(\num{c(U)}) \tto_{m-1} \psi$ also holds for any such $U$'s. 
Hence, $\psi \in Y_{i, m-1}$. 

In either case, we obtain $\psi \in X_{m-1} \cup Y_{i, m-1}$. 
By Claim \ref{DCL3} again, we conclude that $\PR_{g_2}^{\mathrm{R}}(\gn{\psi})$ holds. 
\end{proof}

\begin{cl}\label{DCL5}
For any $\LA$-formula $\varphi$, $\PA \vdash \neg \bigl(\PR_{g_2}^{\mathrm{R}}(\gn{\varphi}) \land \PR_{g_2}^{\mathrm{R}}(\gn{\neg \varphi}) \bigr)$. 
\end{cl}
\begin{proof}
Let $\varphi$ be any $\LA$-formula. 
Since the sentences $\forall x \bigl(\Prov_T(x) \leftrightarrow \Pr_{g_2}^{\mathrm{R}}(x) \bigr)$ and $\neg \bigl(\Prov_T(\gn{\varphi}) \land \Prov_T(\gn{\neg \varphi}) \bigr)$ are provable in $\PA + \neg \exists x(S(x) \land x \neq 0)$, we obtain $\PA + \neg \exists x(S(x) \land x \neq 0) \vdash \neg \bigr(\PR_{g_2}^{\mathrm{R}}(\gn{\varphi}) \land \neg \PR_{g_2}^{\mathrm{R}}(\gn{\neg \varphi}) \bigr)$. 
So, it suffices to prove $\PA + \exists x(S(x) \land x \neq 0) \vdash \neg \bigl(\PR_{g_2}^{\mathrm{R}}(\gn{\varphi}) \land \neg \PR_{g_2}^{\mathrm{R}}(\gn{\neg \varphi}) \bigr)$. 

We argue in $\PA + \exists x(S(x) \land x \neq 0)$: 
Let $m$ and $i \neq 0$ be such that $h(m) = 0$ and $h(m+1) = i$. 
Then, by Claim \ref{DCL2}, $\varphi \notin X_{m-1} \cup Y_{i, m-1}$ or $\neg \varphi \notin X_{m-1} \cup Y_{i, m-1}$. 
By Claim \ref{DCL3}, we conclude that $\neg \bigl(\PR_{g_2}^{\mathrm{R}}(\gn{\varphi}) \land \PR_{g_2}^{\mathrm{R}}(\gn{\neg \varphi}) \bigr)$ holds. 
\end{proof}

\begin{cl}\label{DCL6}
Let $i \in W_n$. 
\begin{enumerate}
	\item For each $n$-choice function $c$, $\PA \vdash S(\num{i}) \to \PR_{g_2}^{\mathrm{R}} \Bigl(\gn{ \bigvee_{i \prec_n V} S(\num{c(V)})} \Bigr)$.
	
	\item For each $V \in \mathcal{P}(W_n)$ such that $i \prec_n V$, $\PA \vdash S(\num{i}) \to \neg \PR_{g_2}^{\mathrm{R}} \Bigl(\gn{\neg \bigvee_{j \in V}S(\num{j})} \Bigr)$. 
\end{enumerate}
\end{cl}
\begin{proof}
1. Let $c$ be any $n$-choice function. 
We proceed in $\PA + S(\num{i})$: 
Let $m$ and $i \neq 0$ be such that $h(m) = 0$ and $h(m+1) = i$. 
Since $(W_n, \prec_n)$ is a standard finite $\MND$-frame, we obtain that $S(\num{c(U)}) \tto_{m-1} \bigvee_{i \prec_n V} S(\num{c(V)})$ for any $U \in \mathcal{P}(W_n)$ with $i \prec_n U$. 
Then, we have $\bigvee_{i \prec_n V} S(\num{c(V)}) \in Y_{i, m-1}$. 
By Claim \ref{DCL3}, $\PR_{g_2}^{\mathrm{R}}\Bigl(\gn{\bigvee_{i \prec_n V} S(\num{c(V)})} \Bigr)$ holds. 

2. Let $V \in \mathcal{P}(W_n)$ be such that $i \prec_n V$. 
We reason in $\PA + S(\num{i})$: 
Suppose, towards a contradiction, that $\neg \bigvee_{j \in V} S(\num{j}) \in X_{m-1} \cup Y_{i, m-1}$. 
We distinguish the following two cases: 

Case 1: $\neg \bigvee_{j \in V} S(\num{j}) \in X_{m-1}$. \\
There exists some $\psi \in P_{T, m-1}$ such that $\psi \tto_{m-1} \neg \bigvee_{j \in V} S(\num{j})$. 
Since $V$ is non-empty, we find some $j_0 \in V$. 
Then, $\neg S(\num{j_0})$ is a t.c.~of $P_{T, m-1}$. 

Case 2: $\neg \bigvee_{j \in V} S(\num{j}) \in Y_{i, m-1}$. \\
There exists an $n$-choice function $c$ satisfying $S(\num{c(U)}) \tto_{m-1} \neg \bigvee_{j \in V} S(\num{j})$ for any $U \in \mathcal{P}(W_n)$ with $i \prec_n U$.  
Then, $S(\num{c(V)}) \tto_{m-1} \neg \bigvee_{j \in V} S(\num{j})$ because $i \prec_n V$. 
Since $S(\num{c(V)})$ is a disjunct of $\bigvee_{j \in V} S(\num{j})$, we have that $\neg S(\num{c(V)})$ is a t.c.~of $P_{T, m-1}$. 

In either case, $h(m) \neq 0$, and this is a contradiction. 
Therefore, $\neg \bigvee_{j \in V} S(\num{j}) \notin X_{m-1} \cup Y_{i, m-1}$. 
By Claim \ref{DCL3}, we conclude $\neg \PR_{g_2}^{\mathrm{R}} \Bigl(\gn{\neg \bigvee_{j \in V}S(\num{j})} \Bigr)$. 
\end{proof}

We define an arithmetical interpretation $f_{g_2}$ based on $\PR_{g_2}^{\mathrm{R}}(x)$ as in the proof of Theorem \ref{UACTMN}. 
Then, we obtain the following claim as in the proof of Claim \ref{CL5}. 

\begin{cl}\label{DCL7}
Let $i \in W_n$ and $A$ be any $\LB$-formula. 
\begin{enumerate}
	\item If $i \Vdash_n A$, then $\PA \vdash S(\num{i}) \to f_{g_2}(A)$. 

	\item If $i \nVdash_n A$, then $\PA \vdash S(\num{i}) \to \neg f_{g_2}(A)$. 
\end{enumerate}
\end{cl}

The first clause of the theorem follows from Claims \ref{DCL1}, \ref{DCL4} and \ref{DCL5}. 
The second clause is proved by using Proposition \ref{Prop:h}.3 and Claim \ref{DCL7}. 
\end{proof}

\begin{cor}[The arithmetical completeness of $\MND$]\label{ACTMND}
For any $\LB$-formula $A$, the following are equivalent: 
\begin{enumerate}
	\item $\MND \vdash A$. 
	\item $A \in \PL(\PR_T^{\mathrm{R}})$ for any Rosser provability predicate $\PR_T^{\mathrm{R}}(x)$ of $T$ satisfying $\M$ such that $T \vdash \neg \bigl(\PR_T^{\mathrm{R}}(\gn{\varphi}) \land \PR_T^{\mathrm{R}}(\gn{\neg \varphi}) \bigr)$ for all $\LA$-formulas $\varphi$. 
\end{enumerate}
Moreover, there exists a Rosser provability predicate $\PR_T^{\mathrm{R}}(x)$ of $T$ satisfying $\M$ such that $\PL(\PR_T^{\mathrm{R}}) = \MND$. 
\end{cor}

\section{Future Work}\label{Sec:FW}

Our proofs of the arithmetical completeness theorems in the present paper are done by embedding $\MN$-models into arithmetic. 
As noted in Remark \ref{Neighborhood}, our semantics based on $\MN$-models is essentially same as monotonic neighborhood semantics. 
Thus, it can be seen that our proofs are done by embedding monotonic neighborhood models, which are not based on relational semantics. 
This brings us to the natural question of whether our argument can be applied to neighborhood semantics in general.
Neighborhood semantics is a semantics for extensions of the logic $\mathsf{EN}$, that is obtained from $\MN$ by replacing the rule \textsc{RM} with \textsc{RE} $\dfrac{A \leftrightarrow B}{\Box A \leftrightarrow \Box B}$, thus we propose the following problem: 

\begin{prob}
Is the logic $\mathsf{EN}$ arithmetically complete with respect to provability predicates $\PR_T(x)$ satisfying the following condition $\mathbf{E}$?
\begin{description}
	\item [$\mathbf{E}$:] If $T \vdash \varphi \leftrightarrow \psi$, then $T \vdash \PR_T(\gn{\varphi}) \leftrightarrow \PR_T(\gn{\psi})$. 
\end{description} 
\end{prob}

In \cite{Kur21}, it is proved that if a provability predicate $\PR_T(\gn{\varphi})$ satisfies $\M$ and $\D{3}$, then there exists an $\LA$-sentence $\varphi$ such that $T \nvdash \neg \bigl(\PR_T(\gn{\varphi}) \land \PR_T(\gn{\neg \varphi}) \bigr)$. 
It follows that for any provability predicate $\PR_T(x)$ of $T$ satisfying $\M$, the non-inclusion $\MNDF \nsubseteq \PL(\PR_T)$ holds. 
However, the condition that $\PR_T(x)$ satisfies $\M$ is sufficient but not necessary for the inclusion $\MN \subseteq \PL(\PR_T)$. 
Then, we propose the following question: 

\begin{prob}\leavevmode
\begin{enumerate}
	\item Is there a provability predicate $\PR_T(x)$ of $T$ such that $\MNDF \subseteq \PL(\PR_T)$ holds?
	\item Furthermore, is there a provability predicate $\PR_T(x)$ of $T$ such that $\MNDF = \PL(\PR_T)$ holds?
\end{enumerate}
\end{prob}

\section*{Acknowledgment}
This work was supported by JSPS KAKENHI Grant Number JP19K14586. 
The authors would like to thank Sohei Iwata, Yuya Okawa, and Hidenori Kurokawa for their helpful comments. 
The authors would also like to thank the anonymous referee for his or her valuable comments and suggestions.

\bibliographystyle{plain}
\bibliography{ref}

\end{document}